\documentclass[a4paper, 10pt]{article}

\input xy
\xyoption{all}

\newcounter{ExacSeq}

\usepackage[normalem]{ulem}%use \sout{} to strike out text
\usepackage{amsmath,amsthm,amssymb}
\usepackage{stmaryrd}
\usepackage[page,header]{appendix}
\usepackage{morefloats}
\usepackage{color}
\usepackage{makeidx}
\usepackage{fancybox}
\usepackage{mathabx}
\usepackage{tikz}
\usetikzlibrary{matrix}
\usepackage{arcs}
\usepackage{hyperref}

\usepackage{yfonts}
\usepackage[T1]{fontenc}

\usepackage[normalem]{ulem}
\usepackage{amsmath}
\newcommand{\stkout}[1]{\ifmmode\text{\sout{\ensuremath{#1}}}\else\sout{#1}\fi}
\makeindex 

\usepackage{latexsym, amsmath, amssymb, amsfonts, amscd}
\usepackage{amsthm}
\usepackage{stmaryrd}
\usepackage{t1enc}
\usepackage[mathscr]{eucal}
\usepackage{indentfirst}
\usepackage{graphicx, pb-diagram}
\usepackage{enumerate}
\usepackage[all,poly,web,knot]{xy}
\usepackage{float}
\restylefloat{figure}
\usepackage{kantlipsum}
\usepackage[shortlabels]{enumitem}

\newcommand{\Title}{Title}

\numberwithin{equation}{section}

{\theoremstyle{definition}\newtheorem{definition}{Definition}[section]
	
	\newtheorem{defititle}[definition]{\Title}
	\newtheorem{notation}[definition]{Notation}

	\newtheorem{remark}[definition]{Remark}

	\newtheorem{ex}[definition]{Example}
	
	\newtheorem{exs}[definition]{Examples}}

\newtheorem{prop}[definition]{Proposition}
\newtheorem{proposition-definition}[definition]{Proposition-Definition}
\newtheorem{lemma}[definition]{Lemma}
\newtheorem{thm}[definition]{Theorem}
\newtheorem{cor}[definition]{Corollary}

\newtheorem{note}[definition]{Note}
\newtheorem*{prop*}{Proposition}
\newtheorem*{theorem*}{Theorem}
\newcommand{\cD}{\mathcal{D}}
\newcommand{\cG}{\mathcal{G}}

\newcommand{\cF}{\mathcal{F}}

\newcommand{\id}{{\hbox{id}}}

\newcommand{\ie}{{\it i.e.}\/ }

\newcommand{\cf}{{\it cf.}\/ }

\newcommand{\vX}{\mathfrak{X}}

\def\gpd{\,\lower1pt\hbox{$\longrightarrow$}\hskip-.24in\raise2pt
	\hbox{$\longrightarrow$}\,}

\renewcommand{\latticebody}{\drop@{ }}

\newcommand{\N}{\ensuremath{\mathbb N}}

\newcommand{\R}{\ensuremath{\mathbb R}}

\newcommand{\g}{\ensuremath{\mathfrak{g}}}

\newcommand{\h}{\ensuremath{\mathfrak{h}}}

\newcommand{\cX}{\mathcal{X}}

\newcommand{\RR}{\ensuremath{\mathbb R}}

\newcommand{\bt}{\mathbf{t}}                  %target
\newcommand{\bs}{\mathbf{s}}                  %source

\newcommand{\rar}[1]{\overset{\rightarrow}{#1}}
\newcommand{\lar}[1]{\overset{\leftarrow}{#1}}
\def\act{\mathbin{\hbox{$<\kern-.4em\mapstochar\kern.4em$}}}
\def\ract{\mathbin{\hbox{$\mapstochar\kern-.3em>$}}}

\def\bb{\boldsymbol{\beta}}

\def\PB(#1,#2,#3,#4){\left\{\begin{matrix}#1&\!\!\!\stackrel{?}{\longrightarrow}&\!\!\!#2\\
		\downarrow&&\!\!\!\downarrow\\
		#3&\!\!\!\stackrel{?}{\longrightarrow}&\!\!\!#4\end{matrix}\right\}}

\def\pb(#1,#2,#3,#4){ \hom(#1 \to #3, #2 \to #4)}

\newcommand{\pd}[1]{\frac{\partial}{\partial #1}} %\pd{x}
\newcommand{\la}{\langle}
\newcommand{\ra}{\rangle}
\begin{document}
	
	%\nbt{Remove} \nrt{Add}
	
	\begin{center}
		{\Large\bf Coordinates for non-integrable Lie algebroids%\icomment{Since we decided to look only at filtrations of (almost) regular foliations, the terminology "filtered singular foliation" is no longer valid, we should think of a different name... Meinrenken uses "singular Lie filtrations" - I think it's a good term, because "singular" refers to the nature of the filtration, not just its top term (for Meinrenken the top term is always $TM$). Also it is quite accurate, I mean "Lie filtration" signifies that we don't just look at the fintration, but the bracket plays a major role as well. So we could call our objects "singular Lie fintration of a (almost regular) foliation". (Also, this way there will be only one terminology in the literature.)} 
			
			% 46L87 Noncommutative differential geometry
			% 58B34 Noncommutative geometry (�  la Connes)
			% 58J22  Exotic index theories
			% 46L45 Decomposition theory for $C^*$-algebras
			% 19K35 Kasparov theory ($KK$-theory)
			% 46L80 $K$-theory and operator algebras (including cyclic theory)
			% 93B18 Linearizations
			% 53C29 Issues of holonomy
			% 53C12 Foliations (differential geometric aspects)
			% 53R30 Foliations; geometric theory
			% 22A25 Representations of general topological groups and semigroups
			% 22A22 Topological groupoids (including differentiable and Lie groupoids)
			\bigskip
			
			{\sc by Iakovos Androulidakis}
		}
		
	\end{center}

	{\footnotesize
		\vskip -2pt National and Kapodistrian University of Athens
		\vskip -2pt Department of Mathematics
		\vskip -2pt Panepistimiopolis
		\vskip -2pt GR-15784 Athens, Greece
		\vskip -2pt e-mail: \texttt{iandroul@math.uoa.gr}
	}
	\bigskip
	%\everymath={\displaystyle}
	
	\begin{center}
	\textit{Dedicated to the memory of my teacher, Kirill C. H. Mackenzie.} %\\ \today
	\end{center}

	\begin{abstract}%\noindent 
	We construct local coordinates for the Weinstein groupoid of a non-integrable Lie algebroid. To this end, we reformulate the notion of bi-submersion in a completely algebraic way and prove the existence of bi-submersions as such for the Weinstein groupoid. This implies that a $C^*$-algebra can be attached to every Lie algebroid.
	\end{abstract}
	
	%\nrt{I'll write my comments in red using {$\backslash$\text{icomment\{...\}}}. Also, any new text that I add (not a comment), will be in red, using {$\backslash$\text{nrt\{...\}}}.}
	
	%\nbt{Bob can write his comments in blue using {$\backslash$\text{bcomment\{...\}}}. Also, any new text that Bob adds (not a comment), will be in blue using {$\backslash$\text{nbt\{...\}}}.} 
	
	%\ngt{Omar can write his comments in green using {$\backslash$\text{ocomment\{...\}}}. Also, any new text that Omar adds (not a comment), will be in green using {$\backslash$\text{ngt\{...\}}}.}
	
	%\nmt{I'm going to use magenta for things that I think could be dropped from the final version.}
	
	\setcounter{tocdepth}{2} %doesn't display subsections in TOC 
	\tableofcontents
	
	\section*{Introduction}
	
	Groupoids are very useful tools to understand the internal and external symmetries of various examples of dynamical systems, such as symplectic and Poisson structures, also foliations. Their power lies in the application of the following functors, which allow us to study a dynamical system from the classical and the quantum viewpoint.
\begin{itemize}
\item The ``Lie'' functor. That is, differentiating a groupoid to its algebroid. (One might also speak about the ``integration'' functor, that is integrating an algebroid to a groupoid).
\item The ``$C^{\ast}$'' functor: That is, attaching a\footnote{Usually one is interested in attaching a pair of $C^*$-algebras to a groupoid, namely the ``full'' and the ``reduced'' one.} $C^{\ast}$-algebra to a groupoid.
\end{itemize}
	
	By and large, groupoids arise from a process similar to the one introduced by Sophus Lie for differential equations, which eventually lead to the notion of a Lie group. In analogy with the case of Lie algebras, one can always integrate a Lie algebroid $A$ over $M$ to a local Lie groupoid, namely a smooth manifold $G$ with two submersions $G \to M$, endowed with a smooth partial multiplication which is defined only near the identity. However, unlike the intuition from Lie's third theorem, having a groupoid with global partial multiplication often  comes at the expense of its smoothness. 
	
	Recall that the work of Almeida and Molino \cite{AM} gives the first counterexample of a transitive Lie algebroid which does not integrate to a Lie groupoid. Mackenzie \cite{MK2} gave the cohomological obstruction to integrability in the transitive case. The work of Cattaneo and Felder \cite{CatFeld} shows that in Poisson geometry we can obtain a source-simply connected topological groupoid with globally defined partial multiplication, but its topology may not be differentiable. In the work of Crainic and Fernandes \cite{CrFeLie} it is shown that such a groupoid exists for any Lie algebroid $A$ and be realised as a quotient of a space of paths on $A$ ($A$-paths), in analogy with the treatment given in \cite{DuistKolk} for Lie algebras. The name ``Weinstein groupoid of $A$'' and the notation $W(A) \gpd M$ for this groupoid, were coined in this work. Moreover, the explicit obstructions to the differentiability of $W(A)$ were given, generalising the work of Mackenzie.  %This groupoid is a quotient of a space of paths on $A$. The explicit obstructions to its differentiability were given by Crainic and Fernandes \cite{CrFeLie}. 
	\begin{sloppypar}
	Topological groupoids arise in foliation theory as well. The holonomy groupoid of a foliation is a topological groupoid as such. It can be defined \cite{AndrSk} for every foliation, even when the dimension of the leaves is not constant (singularities). However, it is a Lie groupoid only when the foliation is ``almost regular'' \cite{DebordJDG}. Zambon \cite{Za18} generalised the construction of the holonomy groupoid to singular versions of Lie subalgebroids as well. Note that the notion of Lie algebroid is relaxed in both \cite{AndrSk} and \cite{Za18}. Namely, the rank is allowed to be a semi-continuous function rather than constant.
	\end{sloppypar}
	%The power of groupoids lies in the following functors, which allow us to study a dynamical system from the classical and the quantum viewpoint.
%\begin{itemize}
%\item The ``Lie'' functor. That is, differentiating a groupoid to its algebroid. (One might also speak about the ``integration'' functor, that is integrating an algebroid to a groupoid).
%\item The ``$C^{\ast}$'' functor: That is, attaching a $C^{\ast}$-algebra to a groupoid.
%\end{itemize}
Both of the ``Lie'' and ``$C^{\ast}$'' functors apply to the category of Lie groupoids. Beyond this category, their application is not always possible: For local Lie groupoids the $C^{\ast}$ functor fails. For general locally compact topological groupoids (with a Haar system), the $C^{\ast}$ functor does apply but the Lie functor is meaningless, unless the topology is differentiable.

Although the topology of the groupoids constructed in \cite{AndrSk}  \cite{Za18} is extremely pathological, both functors can be applied on them (\cf \cite{AnZa20}, \cite[Appendix]{Za18}). The reason for this is that these are \textit{diffeological} groupoids, namely they have well behaved (smooth) local coordinate systems, although their dimension varies. These local coordinates are called \textit{bi-submersions} and were introduced in \cite{AndrSk}.

%The current status of the ``Lie'' and ``$C^{\ast}$'' functors in the case of the Weinstein groupoid $W(A)$ is this: When the Crainic-Fernandes integrability obstructions vanish, $W(A)$ is a Lie groupoid. Otherwise, it is neither clear how to differentiate $W(A)$ explicitly, nor how to attach a $C^*$-algebra to it, although the latter is predicted in \cite{CrFeLie}.

\subsection*{Coordinates for the Weinstein groupoid}

In this note we show that the Weinstein groupoid is a diffeological groupoid as well. In other words, we construct bi-submersions for $W(A)$. This is the crucial step towards applying both functors to $W(A)$, in particular the $C^*$-algebra construction given in \cite[Appendix]{Za18}. 

Our construction starts from the following observations. The anchor map induces a morphism of $C^{\infty}(M)$-modules $\Gamma_c(A) \to \cX_c(M)$ which preserves the Lie brackets. The image $\cF$ of this map is exactly the singular foliation of the base manifold $M$ induced by the Lie algebroid $A$. That is, $(M,\cF)$ is a singular foliation in the sense of \cite{AndrSk} (a locally finitely generated  and involutive submodule of $\cX_c(M)$). We obtain a short exact sequence of $C^{\infty}(M)$-modules: $$0 \to \h \to \Gamma_c(A) \to \cF \to 0$$ The following observations are crucial:
\begin{itemize}
\item The kernel $\h$ induces the foliation by points on $M$ (see \S \ref{sec:extensions}). %Indeed, if $G$ is a local Lie groupoid integrating $A$, then the $C^{\infty}(G)$-module $\overrightarrow{\h} \cap \overleftarrow{\h}$ of right-invariant and left-invariant vector fields corresponding to $\h$ is the foliation of $G$ whose leaves are the isotropy local Lie groups $G_x^x = s^{-1}(x) \cap t^{-1}(x)$. %Whence we can construct bi-submersions of $\h$ from bi-submersions of $\overrightarrow{\h}\cap \overleftarrow{\h}$.
\item The foliation $(M,\cF)$ integrates (in the sense of \cite{AnZa20}) to the holonomy groupoid $H(\cF)$ constructed in \cite{AndrSk}. This becomes a diffeological groupoid using the path-holonomy bi-submersions constructed in \cite{AndrSk}. Recall that the restriction of $H(\cF)$ to a leaf is a Lie groupoid (transitive).
\item An $A$-path projects to a path which stays entirely in a leaf $L$ of the foliation $(M,\cF)$. Moreover, it suffices to consider only ``small'' paths as such. %That is because ``long'' $A$-paths are equivalent ($A$-homotopic) to piecewise smooth $A$-paths with the same end points.
\item Restricting to a leaf $L$ we obtain a short exact sequence of Lie algebroids. The choice of a splitting for this sequence provides an isomorphism of the transitive Lie algebroid $A_L$ with the familiar (\cf \cite{MK2}) transitive Lie algebroid $A^{\cF}_L \oplus \h_L$, where $A^{\cF}_L$ is the Lie algebroid of the  restriction $H(\cF)_L^L$ of $H(\cF)$ to $L$. 
\end{itemize}
Our principal result is that, a bi-submersion of the foliation $(M,\cF)$ and a bi-submersion of the trivial foliation $\h$ can be combined to a new bi-submersion $(Z,s_Z,t_Z)$ of the foliation $(M,\cF)$. This new bi-submersion plays the role of coordinates for $W(A)$, even when $A$ is not integrable. More explicitly, $Z$ is a smooth manifold and $s_Z, t_Z : Z \to M$ are submersions such that: 
\begin{enumerate} 
\item The following $C^{\infty}(Z)$-module of vector fields on $Z$ is involutive (thus it is a singular foliation on $Z$):
\begin{equation}\label{eq:dfnbisubm1}
\widehat{\cF}_Z = \Gamma_c(Z;\ker ds_Z) + \Gamma_c(Z;\ker ds_Z)
\end{equation}
\item Each one of the submersions $s_Z, t_Z$ maps $(Z,\widehat{\cF}_Z)$ to $(M,\cF)$.
\item There is a map $\psi : Z \to W(A)$ which commutes with the source and target maps and carries the singular foliation $(Z,\widehat{\cF}_Z)$ to an appropriate singular foliation on $W(A)$.
\end{enumerate}
The map $\psi$ is crucial because it justifies viewing the bi-submersion $(Z,s_Z,t_Z)$ as local coordinates for $W(A)$. Of course, when $W(A)$ is not a Lie groupoid, there are no vector fields around to make sense of a foliation on $W(A)$. So one needs to make the last statement precise. %To achieve this, we show that the notion of a bi-submersion admits a completely algebraic formulation. 

\subsection*{Reformulation of the notion of bi-submersion}

%The next step is to show that $(Z,s_Z,t_Z)$ is actually a bi-submersion of $W(A)$. To this end, we need to find a map $\psi : Z \to W(A)$ which commutes with the source and target maps, moreover it satisfies a condition analogous to \eqref{eq:dfnbisubm1}. Here is exactly where the difficulty lies: When $A$ is not integrable, the Weinstein groupoid does not have a smooth structure, therefore there are no vector fields associated to it. 

This calls for a complete reformulation of the notion of bi-submersion, so that it can make sense in the case of spaces without a smooth structure. We do this in sections \ref{sec:bisubmersions} and \ref{sec:bisubmintegrable}, namely we show that the involutivity of the foliation $\widehat{\cF}_Z$ can be reformulated to a purely algebraic condition. The easiest way to describe this condition is to look at what happens with Lie groupoids. Recall that a Lie groupoid $\cG \gpd M$ is a bi-submersion for the foliation $(M,\cF)$ defined by its orbits. %in particular the $C^{\infty}(G)$-module of vector fields $\widehat{\cF}_{\cG}=\Gamma_c(\cG;\ker ds_{\cG}) + \Gamma_c(\cG;\ker ds_{\cG})$ defined by the formula \eqref{eq:dfnbisubm1}. 

Now consider the $C^{\infty}(G)$-module of vector fields $\widehat{\cF}_{\cG}=\Gamma_c(\cG;\ker ds_{\cG}) + \Gamma_c(\cG;\ker ds_{\cG})$. Our result is that the involutivity of $\widehat{\cF}_{\cG}$ is equivalent to the existence of the map 
\begin{equation}\label{eq:dfnbisubm2}
\cG \times_s \cG \times_t \cG \to \cG \times_s \cG \times_t \cG, \quad (g,h,\zeta) \mapsto (\zeta,gh^{-1}\zeta,g)
\end{equation}
Notice that this map makes use of the groupoid multiplication. On the other hand, the bi-submersion $(Z,s_Z,t_Z)$ is a local object. So, in order to replace the involutivity of $\widehat{\cF}_Z$ by the localized version of the map \eqref{eq:dfnbisubm2}, we first observe that $(Z,s_Z,t_Z)$ carries a natural local Lie groupoid structure. 

With this formulation in hand, the statement ``$\psi : Z \to W(A)$ carries the singular foliation $(Z,\widehat{\cF}_Z)$ to an appropriate singular foliation on $W(A)$'' is equivalent to the fact that $\psi \times_{s} \psi \times_t \psi$ commutes with the localisation of the map \eqref{eq:dfnbisubm2}. 

In the overall, it is worth noting that, since bi-submersions really provide an equivalent definition of a foliation, our reformulation allows to extend the notion of a foliation beyond smooth manifolds. We are not concerned with the study of such cases here, but it would be interesting to explore the range of applicability of the $C^{\ast}$-functor, using this generalised notion of bi-submersion.

\subsection*{Construction of the map $\psi$}

The map $\psi : Z \to W(A)$ is constructed making use of the following result from \cite{DuistKolk}: Every connected and simply connected Lie group $H$ (finite dimensional) is the quotient of the Banach manifold $P(\h)$ of paths in its Lie algebra $\h$ by a certain kind of homotopy. Recall that the Weinstein groupoid $W(A)$ is constructed following the same principle. Namely, it is the quotient of the Banach manifold $P(A)$ of $A$-paths by $A$-homotopy. 

The local group structure of $Z$ allows us to view an element $z$ of $Z$ near the identity as a homotopy class of a pair of paths in appropriate Lie algebras: the first is a path on a fiber of $A^{\cF}_{L}$ and the second is a path on the fiber of $\h_L$. Then, the splitting of the transitive Lie algebroid $A_L$ to $A^{\cF}_{L}\oplus\h_L$ allows us to translate the previous pair of paths to an $A$-path of the Lie algebroid $A$. We then define $\psi(z) \in W(A)$ as the $A$-homotopy class of the latter $A$-path.
	\\ \\ 
	\textbf{Acknowledgements} I would like to thank Marco Zambon for several discussions and suggestions. Examples \ref{exs:bisubmersion}(b), \ref{ex:marco1} and \ref{ex:marco2} are by him.

\section{Foliations and Lie algebroids}\label{sec:preliminaries}

\subsection{Foliations}\label{sec:preliminaries1}

Let $M$ be a smooth manifold without boundary, with finite dimension. Throughout the sequel, we will use the following notation:
\begin{itemize}
\item[-] $\cX(M)$ ($\cX_c(M)$) is the $C^{\infty}(M;\R)$-module of smooth (compactly supported) vector fields on $M$.
\item[-] If $E \to M$ is a vector bundle, $\Gamma E$ ($\Gamma_c E$) is the $C^{\infty}(M;\R)$-module of smooth (compactly supported) sections of $E$.
\end{itemize}
Let us recall the following from \cite{AndrSk}.

\begin{enumerate}
\item Let $\cD$ be a $C^{\infty}(M;\R)$-submodule $\cD$ of $\cX_c(M)$, $U$ a smooth manifold and $f : U \to M$ a smooth map. We denote $f^{\ast}\cD$ the following $C^{\infty}(U;\R)$-submodule of $\Gamma_c(f^{\ast}TM)$: $$f^{\ast}\cD = \{\sum_{i=1}^{k}g_i\cdot(X_i \circ f) : k \in \N, g_i \in C^{\infty}_c(U;\R), X_i \in \cD\}$$
\item Put $f^{-1}(\cD)$ the following $C^{\infty}(U;\R)$-submodule of $\cX_c(U)$: $$f^{-1}(\cD) = \{X \in \cX_c(U) : df \circ X \in f^{\ast}\cD\}$$ When $U$ is an open subset of $M$ and $\iota_U : U \to M$ the inclusion map we have $\iota_U^{\ast}(\cD) = \iota_U^{-1}(\cD) = \{X \in \cD : supp(X)\subset U\}$. We denote the $C^{\infty}(U;\R)$-submodule $\iota_U^{\ast}(\cD)$ by $\cD|_{U}$.
\item We say that a finite family of vector fields $X_i \in \cD$, $i=1,\ldots,k$ generates the module $\cD$ over $U$ if $$\cD|_{U} = C^{\infty}_{c}(U;\R)X_1 |_U + \ldots C^{\infty}_{c}(U;\R)X_k |_U$$ We say that this family of vector fields generates $\cD$ at the point $p \in M$ if it generates $\cD$ on an open neighborhood of $p$.
\item A \textit{generalised smooth distribution} $(\cD,M)$ is a $C^{\infty}(M,\R)$-submodule $\cD$ of $\cX_c(M)$, which is finitely generated at every point $p \in M$. (Note that the minimal number of generating vector fields may differ from point to point.)
\item Let $f : U \to M$ be a submersion In \cite[Prop. 1.10]{AndrSk} it is shown that if $(M,\cD)$ is a generalised smooth distribution then $f^{-1}(\cD)$ is a generalised smooth distribution as well.
\item A generalized smooth distribution $(M,\cD)$ is called \textit{regular} if there is a locally trivial vector subbundle $E \subseteq TM$ such that $\cD = \Gamma_c(E)$. Otherwise it is called \textit{singular}.
%\item In \cite[Prop. 1.5]{AndrSk} it is shown that there is an open dense subset $\cD_{reg} \subseteq M$ such that the restriction of the generalised smooth distribution $(M,\cD)$ to this set is regular.
\item A generalized smooth distribution $(M,\cD)$ such that $\cD$ is closed with respect to the Lie bracket of vector fields, is called a \textit{foliation}. We use the notation $(M,\cF)$ instead of $(M,\cD)$ in this case. 
\item The Stefan-Sussmann theorem shows that a foliation $(M,\cF)$ as such, induces a partition of $M$ to immersed submanifolds, called leaves. Throughout the sequel we will assume that every leaf $L$ is \textit{embdedded}. This assumption is only for the simplicity of the exposition. All of our results can be proven without difficulty for immersed leaves as well.
\end{enumerate}

We will also use the following lemma in the sequel, which was proven in \cite[Lemma 2.2]{androulidakis2013a} for foliations. Its proof is the same as the one in \cite{androulidakis2013a}.

\begin{lemma}\label{lem:submfol}
Let $f : U \to M$ be a surjective submersion with connected fibres. Let $(U,\cD_U)$ be a generalised smooth distribution which satisfies: 
\begin{itemize}
\item $\Gamma_c(U;\ker df) \subseteq \cD_U$ and  
\item $[\Gamma_c(U;\ker df),\cD_U] \subseteq \cD_U$. 
\end{itemize}
Then there is a unique generalised smooth distribution $(M,\cD_M)$ such that $$f^{-1}(\cD_M) = \cD_U.$$ 

Conversely, for any generalised smooth distribution $(M,\cD_M)$ the generalised smooth distribution $(U,f^{-1}(\cD_M))$ satisfies the two properties above.
\end{lemma}
  
Notice that, when $(U,\cD_U)$ is a foliation, the second condition is implied by the first one, and in addition $(M,\cD_M)$ is a foliation.

\subsubsection{Fibers of a foliation}

Let $(M,\cF)$ be a foliation, fix a point $x$ in $M$ and put $L$ the leaf at $x$. Recall from \cite{AndrSk} that we can attach the following two vector spaces to $x$:
\begin{itemize}
\item The tangent space $T_x L$ at the point $x$ of the leaf $L$ at $x$;
\item The quotient $\cF_x = \frac{\cF}{I_{x}\cF}$, where $I_{x} = \{f \in C^{\infty}(M) : f(x)=0\}$. It has finite dimension because the module $\cF$ is locally finitely generated. We call $\cF_x$ the \textit{fiber} of $\cF$ at $x$.
\end{itemize}
Evaluation defines a surjective linear map $\cF_x \to T_x L$, whence we have a short exact sequence: $$0 \to \g_x \to \cF_x \to T_x L \to 0$$ The kernel $\g_x$ is a Lie algebra at every $x \in M$. The following hold:
\begin{enumerate}
\item The vector fields $X_1,\ldots,X_m \in \cX_c(M)$ are a minimal generating set of $\cF$ in a neighborhood of $x$ if and only if their classes $[X_1],\ldots,[X_m]$ are a basis of $\cF_x$. The proof of this result is exactly as in \cite[Prop. 1.5]{AndrSk}.
\item The dimension of $\cF_x$ is upper semi-continuous and it is not bound by the dimension of the manifold. We have $dim(\cF_x) = dim(\cF_y)$ for every $y \in L$.
\item The leaf $L$ is regular if and only if $dim(\cF_x) = dim(T_x L)$. Equivalently, when $\g_x$ vanishes.
\item There is an open dense subset $M_{reg} \subseteq M$ where $dim(\cF_x) = dim(T_x L)$ for every $x \in M_{reg}$, whence $M_{reg}$ is saturated by regular leaves. The restriction $\cF_{M_{reg}}$ is a projective module, namely it is the module of sections of a vector bundle (constant rank).
\item Put $A^{\cF}_L = \cup_{x \in L}\cF_x$. This is a transitive Lie algebroid over $L$ with anchor defined by the evaluation map (\cf \cite{androulidakis2013a}). Its $C^{\infty}(L)$-module of sections is $\frac{\cF}{I_L\cF}$, where $I_{L} = \{f \in C^{\infty}(M) : f|_{L}=0\}$. We obtain the short exact sequence 
\begin{equation}\label{eq:extAL}
0 \to \g_L \to A^{\cF}_L \to TL \to 0 
\end{equation}
where $\g_L$ is the bundle of Lie algebras whose fibers are isomorphic to $\g_x$.
\item The Lie algebroid $A^{\cF}_L$ is integrable. The restriction $H(\cF)_L^L$ of the holonomy groupoid of $\cF$ to $L$ is a Lie groupoid and differentiates to $A^{\cF}_L$ (\cf \cite{Debord2013}).
\end{enumerate}

\subsubsection{Bi-submersions}\label{sec:bisubmgeneral}

In \cite[Prop. 2.4]{androulidakis2013a} it was shown that the notion of a foliation $(M,\cF)$ is equivalent to the notion of \textit{bi-submersion} introduced in \cite{AndrSk}. Let us recall the latter:

\begin{definition}\label{dfn:bisubmfol}
A bi-submersion of the foliation $(M,\cF)$ is a triple $(U,\bs,\bt)$ where $\bs : U \to M$ and $\bt : U \to M$ are submersions which satisfy: 
\begin{equation}\label{eq:bisubm}
\bs^{-1}(\cF)={\bt}^{-1}(\cF) = \Gamma_c(U;\ker d\bs) + \Gamma_c(U;\ker d\bt)
\end{equation}
\end{definition}

The primary examples are \textit{path-holonomy} bi-submersions. These are the building blocks of the holonomy groupoid construction given in \cite{AndrSk}. Let us recall their construction.
\begin{ex}\label{ex:pathol}
Choose $X_1,\ldots,X_m \in \cF$ such that their classes $[X_1],\ldots,[X_m]$ are a basis of the fiber $\cF_x$. Consider the projection map $\bs : M \times \R^m \to M$. There is an open neighbourhood $U$ of $(x,0)$ in $M \times \R^m$ where the following map is defined: $$\bt(y,\lambda_1,\ldots,\lambda_m) = exp(\sum_{i=1}^m \lambda_i X_i)(y)$$ The triple $(U,\bt,\bs)$ is a bi-submersion. It is called the \textit{minimal path-holonomy bi-submersion} of $(M,\cF)$ at the point $x$. Of course we can apply this construction starting with local generators of $\cF$ whose number is larger than $dim(\cF_x)$. All the bi-submersions constructed in this way are called \textit{path-holonomy} bi-submersions.
\end{ex}

Let $(U,\bs,\bt)$ be a bi-submersion of the foliation $(M,\cF)$. The following are discussed in \cite{AndrSk}:
\begin{enumerate}
\setcounter{enumi}{8}
\item A \textit{bisection} of $(U,\bs,\bt)$ is a closed submanifold $\bb$ of $U$ such that the restrictions $\bs|_{\bb}$ and $\bt|_{\bb}$ are local diffeomorpisms. We say that $\bb$ is a bisection at the point $u \in U$ if $u \in \bb$.
\item If $\bb$ is a bisection, the map $\varpi_{\bb} = \bt|_{\bb} \circ (\bs|_{\bb})^{-1} : \bs(\bb) \to \bt(\bb)$ is a local diffeomorphism of $M$ such that $(\varpi_{\bb})_{\ast}(\cF|_{\bs(\bb)})=\cF|_{\bt(\bb)}$. If the bisection $\bb$ is at $u \in U$, we say that $u$ carries the local diffeomorphism $\varpi_{\bb}$. On the other hand, let $\varpi$ be a local diffeomorphism of $M$ which preserves the module $\cF$ (appropriately restricted to the domain and image of $\varpi$). We say that the point $u \in U$ carries $\varpi$ iff there exists a bisection $\bb$ at $u$ such that $\varpi_{\bb}=\varpi$.
\item Let $(V,\bs_V,\bt_V)$ be another bi-submersion of $(M,\cF)$. If the points $u \in U$ and $v \in V$ carry the same local diffeomprhism $\varpi$ as in the previous item, then there exist open neighbourhoods $\tilde{U}_u$ of $u$ and $\tilde{V}_v$ of $v$ and a smooth map $\varpi_{u,v} : \tilde{U}_u \to \tilde{V}_v$ such that $\bs_V \circ \varpi_{u,v} = \bs$ and $\bt_V \circ \varpi_{u,v}=\bt$.  
\item The \textit{inverse} bisubmersion of $(U,\bs,\bt)$ is $(\overline{U},\overline{\bs},\overline{\bt})$ where $\overline{U}=U$, $\overline{\bs}=\bt$ and $\overline{\bt}=\bs$. If $u \in U$ carries the local diffeomorphism $\varpi$ with respect to the bi-submersion $(U,\bs,\bt)$ then $u$ carries $\varpi^{-1}$ with respect to the converse bi-submersion. 
\end{enumerate}

\subsubsection{Holonomy groupoid}\label{sec:holgpd}

In \cite{AndrSk} it was shown that items (i), (j), (k), (l) above define an equivalence relation. The associated quotient is the holonomy groupoid $H(\cF) \gpd M$. Hence bi-submersions of $(M,\cF)$ can be viewed as the coordinates of the holonomy groupoid. Let us recall the following from \cite{AS4} about this groupoid.
\begin{enumerate}
\item If $V$ is an open subset of $M$ then the holonomy groupoid of the restriction $\cF_V$ is the $s$-connected component of the restriction $H(\cF)_V^V = \{z \in H(\cF) : t(z) \in V \text{ and } s(z) \in V\}$.
\item If the open subset $V$ is saturated, then $H(\cF_V) = H(\cF)_V^V$.
\item Let $f : V \to M$ be a smooth map transverse to $\cF$ (\cf \cite[\S 1.2.3]{AndrSk}). Then $H(f^{-1}(\cF))$ is the $s$-connected component of $$H(\cF)^f_f = \{(k_1,z,k_2) \in K \times H(\cF) \times K : t(z)=f(k_1) \text{ and } s(z)=f(k_2)\}$$
\item If moreover $f : V \to M$ is a submersion with connected fibers and its image is saturated, then $H(f^{-1}(\cF)) = H(\cF)_f^f$.
\item Let $K$ be a saturated locally closed subset of $M$. Then $K$ is open in its closure $\overline{K}$ and the closed subset $Y = \overline{K}\setminus K$ is saturated. Consider the open $V = M \setminus (\overline{K}\setminus K)$.  It follows that $H(\cF_V)$ is the $s$-connected component of $H(\cF)_V^V$.
\end{enumerate}

\subsection{Lie algebroids}\label{sec:preliminaries2}

Let $A \to M$ be a Lie algebroid (not necessarily integrable). Put $\tilde{\rho} : A \to TM$ its anchor map and  $(M,\cF)$ the underlying foliation, namely $\cF$ is the module of vector fields generated by $\tilde{\rho} \circ \sigma$ for every compactly supported section $\sigma \in \Gamma_c A$. 

\subsubsection{Exact sequences}\label{sec:extensions}

Now fix a point $x \in M$ and let $V$ be an open neighbourhood of $x$ in $M$. The anchor map $\tilde{\rho}$ induces a morphism of $C^{\infty}(V)$-modules $\rho_V : \Gamma_c(V;A) \to \cF_V$ which is onto. Whence we obtain a short exact sequence of $C^{\infty}(V)$-modules 
\begin{equation}\label{eq:extV}
0 \to \h_V \to \Gamma_c(A_V) \stackrel{\rho_V}{\longrightarrow} \cF_V \to 0
\end{equation}
The kernel $\h_V$ comprises sections $\sigma \in \Gamma_c(V;A)$ such that $\rho \circ \sigma =0$. We have: 
\begin{enumerate}
%\item Since $\Gamma_c(V;A)$ is a free module and $ \cF_V$ is finitely generated, the ideal $\h_V = \ker \rho_V$ is finitely generated. 
\item The Lie bracket of $\Gamma_c(V;A)$ restricts to $\h_V$; in fact, $\h_V$ is a Lie-Rinehart algebra with the zero anchor map.
%\item Let $G \gpd M$ be a local Lie groupoid such that $A=AG$. The ideal $\h_V$ corresponds to the singular foliation $(G_V^V,\overrightarrow{\h_V} \cap \overleftarrow{\h_V})$, where $\overrightarrow{\h_V}\cap \overleftarrow{\h_V}$ is the $C^{\infty}(G_V^V)$-module generated by the right-invariant and left-invariant vector fields of $G_V^V$ corresponding to the generators $\xi_1,\ldots,\xi_k$ of $\h_V$. The leaves of this foliation are the isotropy local Lie groups $G_x^x = s^{-1}(x) \cap t^{-1}(x)$.
\item The fiber $\h_x = \frac{\h_V}{I_x \h_V}$ is a Lie algebra: Indeed, since the anchor map vanishes on sections in $\h_V$, the Leibniz rule for the Lie bracket $[\cdot,\cdot]_A$ of $A$ implies that the formula $[\langle \xi \rangle,\langle \eta \rangle] = \langle [\xi,\eta]_A \rangle$ is well defined for all $\xi,\eta \in \h_V$ (our notation $\langle \xi \rangle$ stands for the class in $\h_x$ of $\xi \in \h_V$). For every other point $y \in L$, the Lie algebra $\h_y$ is isomorphic to $\h_x$.
\item The anchor map satisfies $\rho_V(I_x A_V) \subseteq I_x \cF_V$. Recall that the fiber $A_x$ is the quotient $\frac{\Gamma_c(V;A)}{I_x \Gamma_c(V;A)}$. Whence, we have an exact sequence of vector spaces 
\begin{equation}\label{eq:extx}
0 \to \h_x \to A_x \to \cF_x \to 0
\end{equation}
This implies that the dimension of $\h_x$ is bounded above by the rank of the vector bundle $A$. In fact, the dimension of $\h_x$ is lower semi-continuous. 
\item The sections $\nu_1,\ldots,\nu_k$ in $\h_V$ are a minimal generating set of $\h_V$ if and only if $[\nu_1],\ldots,[\nu_k]$ is a basis of $\h_x$. The proof of this result is exactly as in \cite[Prop. 1.5]{AndrSk}. It follows that the module $\h_V$ is finitely generated.
\item If $\nu_1,\ldots,\nu_k$ is a minimal generating set of $\h_V$, then $\rho_V(\nu_i)=0$ for all $i = 1,\ldots,k$. In this sense, we say that the module $\h_V$ induces the foliation by points on $V$.
\item We also have $\rho(I_{L\cal V} A_V) \subseteq I_{L\cap V} \cF_V$, so we get an extension of transitive Lie algebroids 
\begin{equation}\label{eq:extL}
0 \to \h_{L\cap V} \to A_{L\cap V} \stackrel{\rho}{\longrightarrow} A^{\cF}_{L\cap V} \to 0
\end{equation} 
The Lie algebra bundle $\h_{L \cap V}$ is a totally intransitive Lie algebroid, namely it is equipped with the zero anchor map. Again, it induces the foliation by points on $L \cap V$.
\item The Lie algebra bundle $\h_{L\cap V}$ sits inside the isotropy Lie algebra bundle $\g^A_{L\cap V} = \ker(\tilde{\rho}_{L\cap V})$ of the transitive Lie algebroid $A_{L\cap V}$. The quotient is the kernel $\g_{L\cap V}$ of extension \eqref{eq:extAL}.
\item The semicontinuity of the dimension of $\h_x$ implies that the set $M_A$ of points $x$ in $M$ such that $\h_x$ \textit{does \underline{not}} vanish, is saturated and open. Its complement $M_A^c = M \setminus M_A$ is saturated and closed. Consider the inclusion $\iota : M_A^c \to M$ and put $\cF_{M_A^c}=\iota^{-1}(\cF)$. Then $(M_A^c,\cF_{M_A^c})$ is a singular foliation. Its leaves are the leaves of $\cF$ which lie in $M_A^c$.
\item For every leaf $L$ in $M_A^c$ we have $A_L = A^{\cF}_L$. Therefore, the restriction of $A$ to a every leaf $L$ in $M_A^c$ is an integrable transitive Lie algebroid. The restriction $A_{M_A^c}$ is the union of $A^{\cF}_L$ for all leaves $L$ in $M_A^c$. %Although the boundary $\partial M_A$ is a closed subset of $M$, the second integrability obstruction in \cite{CrFeLie} shows that the algebroid $A_{\partial M_A}$ is not always integrable.
\item %The set $M_A^c$ is a closed saturated subset of $M$. 
The dimension of leaves in $M_A^c$ may not be constant, but since $\h_x = 0$ for every $x \in M_A^c$ we have $A_x = A^{\cF}_x = \cF_x$ for every $x \in M_A^c$. This shows that the fibers of the foliation $(M_A^c,\cF_{M_A^c})$ have constant rank. Whence this foliation is almost regular.  
\item Since the holonomy groupoid of an almost regular foliation is always a Lie groupoid (\cf \cite{Debord2013}), it follows  that $H(\cF_{M_A^c})$ is a Lie groupoid. The Lie algebroid of this groupoid is $A^{\cF}_{M_A^c}$, therefore $A_{M_A^c} = A^{\cF}_{M_A^c}$ is an integrable Lie algebroid. It follows that the restriction $W(A)_{M_A^c}^{M_A^c}$ of the Weinstein groupoid is a Lie groupoid. It is the monodromy groupoid of the foliation $(M_A^c,\cF_{M_A^c})$.
\item The relation between the sets $\partial M_{reg}$ and $M_A^c$ is controlled by the short exact sequence of Lie algebras $$0 \to \h_x \to \g^A_x \to \g_x \to 0$$
\end{enumerate}

\begin{exs}
\begin{enumerate}
\item Consider the transformation Lie algebroid of the action of $SL(2, \R)$ on $\R^2$. Then $\h_x$ is 1-dimensional at every $x \in \R^{2} \setminus \{0\}$ and $\h_0 = 0$.
\item Recall \cite[Ex. 3.8]{AZ3} the Lie algebroid defined by the Lie-Poisson structure on $\mathfrak{su}(2)^{\ast} \cong \R^3$. The (symplectic) foliation of this algebroid is not almost regular. However, $\h_x = \R$ at every $x \in \R^{3} \setminus \{0\}$ and $\h_0 = 0$.
\end{enumerate}
\end{exs}

\subsubsection{Splittings}

Now choose a splitting $\sigma : A^{\cF}_{L\cap V} \to A_{L\cap V}$ of the exact sequence of vector bundles \eqref{eq:extL}. We have:
\begin{itemize}
\item The splitting $\sigma$ determines uniquely a linear map $j : A_{L \cap V} \to \h_{L\cap V}$ such that $j\circ q + \tilde{\rho}\circ \sigma = \id$. The map $\sigma\oplus j : A_{L\cap V} \to A^{\cF}_{L\cap V} \oplus \h_{L \cap V}$ is an isomorphism of vector bundles.
\item Since the leaf $L$ is embedded, the bundles $A^{\cF}_{L\cap V}$ and $\h_{L \cap V}$ are trivial, namely $A^{\cF}_{L\cap V} = (L \cap V) \times \cF_x$ and $\h_{L \cap V} = (L \cap V) \times \h_x$. Whence $A_{L\cap V}$ is also the trivial vector bundle $(L \cap V) \times (\cF_x \times \h_x)$.
\item The map $\sigma\oplus j$ becomes an isomorphism of Lie algebroids when the sections of $A^{\cF}_{L\cap V} \oplus \h_{L \cap V}$ are endowed with the following Lie bracket%\icomment{Do I need to use a connection in this formula?}: 
\begin{eqnarray}\label{eqn:bracket}
[X\oplus V,Y\oplus W] = [X,Y] \oplus \{X(W)-Y(V) + [V,W] - R_{\sigma}(X,Y)\}
\end{eqnarray}
for every $X, Y \in \Gamma(L \cap V;A^{\cF}_{L \cap V})$ and $V, W \in \Gamma(L \cap V;\h_{L \cap V})$. Here $R_{\sigma} : A^{\cF}_{L \cap V} \times A^{\cF}_{L \cap V} \to \h_{L \cap V}$ is the curvature of $\sigma$.
\end{itemize}

%\subsubsection{The Weinstein groupoid}

\subsubsection{Small $A$-paths  and splittings}\label{sec:Apathspl}

%This is justified because every piecewise smooth $A$-path is $A$-homotopic with a smooth $A$-path (\cf \cite{CrFeLie}). 

Recall from \cite{CrFeLie} that the Weinstein groupoid $W(A) \gpd M$ is the quotient of the Banach manifold $P(A)$ of $A$-paths by the appropriate notion of homotopy, called $A$-homotopy. We will only use ``small'' $A$-paths here, this means $A$-paths which live in $A_V$, where $V$ is an open neighbourhood of $x$ in $M$. Let us recall their definition. 

An $A_V$-path is a smooth map $\alpha : [0,1] \to A_V$ such that for every $s \in [0,1]$ we have 
\begin{eqnarray}\label{eqn:Apath}
\rho \circ \alpha(s) = \frac{d}{dt}|_{t=s}(\pi \circ \alpha)(t)
\end{eqnarray} 
We will denote the base path $\pi \circ \alpha$ by $\gamma : [0,1] \to V$. We assume $\gamma(0)=x$. Equation \eqref{eqn:Apath} implies that the path $\gamma$ stays in the leaf $L$, namely  $\gamma(t) \in L \cap V$ for  all $t \in [0,1]$. Whence, it is also an $A_{L\cap V}$-path. We denote $P_0(A)$ the Banach manifold of $A$-paths such that $\alpha(0)$ is an element in the zero section.

In view of the exact sequence \eqref{eq:extL}, equation \eqref{eqn:Apath} shows that small $A$-paths can be characterised in the following way:
\begin{prop}\label{prop:Apath1}
A smooth path $\alpha : [0,1] \to A_V$ is an $A_V$-path whose base path lives in the leaf $L$ if and only if $\rho\circ \alpha$ is an $A^{\cF}_{L\cap V}$-path.
\end{prop}
The proof is a straightforward application of the definitions, so we omit it. The following observations are important:

\begin{itemize}
\item Since $L$ is an embedded leaf, $L \cap V$ is an open subset of $L$ so the bundle $A^{\cF}_{L \cap V}$ is the trivial bundle $(L \cap V) \times \cF_{x}$. Therefore the path $\rho \circ \alpha$ has the form $\gamma \times \rho(\alpha) : [0,1] \to (L \cap V) \times \cF_{x}$. If $\alpha(0)$ is a point in the zero section of $A_{L\cap V}$ then $\rho(\alpha) \in P_0(\cF_x)$. (Note that here we view $\cF_x$ as an abelian Lie algebra. This is justified in Ex. \ref{ex:add}.)

\item On the other hand, let us look at the path $j \circ \alpha : [0,1] \to \h_{L\cap V}$. The bundle $\h_{L\cap V}$ is also trivial, \ie isomorphic to $(L \cap V)\times\h_x$. So the path $j \circ \alpha$ can be written as $\gamma \times j(\alpha) : [0,1] \to (L \cap V)\times\h_x$, where $\gamma$ is the base path of $\alpha$ and $j(\alpha)$ is a path in the Lie algebra $\h_x$. Again, if $\alpha(0)$ is a point in the zero section of $A_{L\cap V}$ then $j(\alpha) \in P_0(\h_x)$. %Recall that $\h_{L\cap V}$ is a totally intransitive Lie algebroid, namely it is considered as a Lie algebroid with the zero anchor map. In other words, it induces the foliation by points on $L \cap V$. Now, for every $y \in L \cap V$, put $c_y$ the constant path at $y$.  So, for every $y \in L \cap V$ the path $\alpha_y = c_y \times j(\alpha)$ is an $\h_{L\cap V}$-path.

\item Note that the path $j(\alpha) \in P_0(\h_x)$ is also a path in $\h_{L\cap V}$. If we consider $\h_{L\cap V}$ as a totally intransitive Lie algebroid, then $j(\alpha)$ is an $\h_{L\cap V}$-path over the constant path at zero. This is natural, since totally intransitive means that $\h_{L\cap V}$ is endowed with the zero anchor map, which induces the foliation by points of $L \cap V$.

\item It follows that the isomorphism $\sigma\oplus j$ trivialises the vector bundle $A_{L \cap V}$ as $(L \cap V) \times (\cF_x \times \h_x)$. It follows that every $A_{L\cap V}$-path $\alpha$ corresponds to a path $\gamma \times \rho(\alpha) \times j(\alpha)$, where $\gamma$ is a path in $L\cap V$, $\rho(\alpha) \in P_0(\cF_x)$ and $j(\alpha) \in P_0(\h_x)$ In other words, $\rho(\alpha) \times j(\alpha)$ is an element of $P_0(\cF_x \times \h_x)$, where $\cF_x \times \h_x$ is the cartesian product Lie algebra arising from the abelian Lie algebra $\cF_x$ and the Lie algebra $\h_x$. 
\end{itemize}

%It follows that, through the isomorphism $\sigma\oplus j$, every $A_{V}$-path $\alpha$ whose base path stays in the leaf $L$ corresponds to a pair $\alpha^{\cF} \oplus \alpha^{\h}$, where $\alpha^{\cF}$ is an $A^{\cF}_{L\cap V}$-path and $\alpha^{\h}$ is an $\h_{L\cap V}$-path. Note that the vector bundle $A^{\cF}_{L\cap V}$ is trivial, namely $A^{\cF}_{L\cap V} = (L \cap V) \times \cF_x$. So every $A^{\cF}_{L\cap V}$-path can be written $\alpha^{\cF} = (\gamma,\alpha_x)$, where $\gamma$ is the base path of $\alpha^{\cF}$ and $\alpha_x$ is an element of the Banach manifold $P_0(\cF_x)$ of smooth paths on the abelian Lie algebra $\cF_x$ which start at the origin.

In \S \ref{sec:pathhol} we'll analyse the $A$-paths of the Lie algebroid $A^{\cF}_{L \cap V}$. To this end, fix a choice of vector fields $\xi_1,\ldots,\xi_n$ in $\cF_V$ such that $\langle\xi_1\rangle,\ldots,\langle\xi_n\rangle$ is an orthonormal basis of the vector space $\cF_x$. Orthonormality here makes sense up to the choice of a smooth fiberwise inner product on the original Lie algebroid $A$. Indeed, it was shown in \cite{AK1} that such a choice gives rise to a family of inner products $(\cdot,\cdot)_y$ on $\cF_y$ for every $y \in V$, which is smooth in a well defined way.

\begin{prop}\label{prop:Apath2}
Let $(y,\alpha_x)$ be an element of $(L \cap V) \times P_0(\cF_x)$. Then there is a unique $A^{\cF}_{L \cap V}$-path $\alpha^{\cF} = (\gamma_y,\alpha_x)$ such that $y=\gamma_y(0)$. The map $(y,\alpha_x) \mapsto (\gamma_y,\alpha_x)$ is a bijection.
\end{prop}
\begin{proof}
For every $t \in [0,1]$ there are unique $\lambda_1(t),\ldots\lambda_n(t) \in \R$ such that $\alpha_x(t) = \sum_{i=1}^n \lambda_i(t)\langle \xi _i\rangle$. Orthonormality implies that $\lambda_i(t) = (\alpha_x(t),\langle \xi_i \rangle)_{x}$ is a smooth function. Consider the time-dependent vector field $X^t = \sum_{i=1}^n \lambda_i(t)\rho(\xi_i)$ and put $\gamma_y(t) = exp(X^t)(y)$.
\end{proof}

\begin{remark}\label{rmk:Apath}
%\begin{enumerate}
%\item 
Recall the homotopy relation defined in \cite[Prop. 1.13.4]{DuistKolk}. It is easy to see that it corresponds to the notion of $A$-homotopy defined in \cite{CrFeLie}. The map defined in proposition \ref{prop:Apath2} sends paths in $P_0(\cF_x)$ which are homotopic in this sense, to $A^{\cF}_{L\cap V}$-homotopic paths.%\icomment{Does this need a proof? To me it's obvious from the Marius-Rui paper...}
%\item The restriction $W(A)_{\partial M_A}^{\partial M_A}$ is the universal cover of $H(\cF)_{\partial M_A}^{\partial M_A}$\icomment{Is this obvious?}. That is, the monodromy groupoid of the almost regular foliation $(\partial M_A, \cF_{\partial M_A})$.
%\end{enumerate}

Moreover, consider paths $\zeta_i = \gamma_i \times \rho(\alpha_i) \times j(\alpha_i)$, $i=1,2$ such that $\gamma_1$ and $\gamma_2$ are homotopic rel end points and the paths $\rho(\alpha_1) \times j(\alpha_1)$, $\rho(\alpha_2) \times j(\alpha_2)$ in $P_0(\cF_x \times \h_x)$ are homotopic in the sense of  \cite[Prop. 1.13.4]{DuistKolk}. Then the $A_{L \cap V}$-paths $\sigma \oplus j(\zeta_1)$ and $\sigma \oplus j(\zeta_2)$ are  $A_{L \cap V}$-homotopic.
\end{remark}

\section{Bi-submersions}\label{sec:bisubmersions}

Here we explain better the notion of bi-submersion and reformulate it. Note that the bi-submersions we discuss here are slightly more general than Definition \ref{dfn:bisubmfol}. Namely, we allow the submersions $\bs, \bt$ to take values in two different manifolds $M$ and $N$ respectively.

%We start with a variation of \cite[Lemma 2.2]{androulidakis2013a}:\mcomment{I am not satisfied with the proof there, I have another one which I think is clearer and more complete}\mcomment{Maybe we should just put the version for foliations, as in \cite{androulidakis2013a}}
%\begin{lemma}\label{lem:submfol}
%Let $f : P \to M$ be a surjective submersion with connected fibres. Let $\cF_P$ be a submodule of $\vX_c(P)$. If $\Gamma_c(P;\ker df) \subseteq \cF_P$ and the Lie bracket with elements of $\Gamma_c(P;\ker df)$ preserves $\cF_P$, then there is a unique submodule $\cF_M$ of $\vX_c(M)$ with $f^{-1}(\cF_M) = \cF_P$. Conversely, for any submodule $\cF_M$ of $\vX_c(M)$,  the submodule $f^{-1}(\cF_M)$ of  $\vX_c(P)$ satisfies the two properties above.
%\end{lemma}
  
%Notice that, when $\cF_P$ is a foliation, the second condition is implied by the first one, and in addition $\cF_M$ is a foliation.

\subsection{Definition of bi-submersion}
  
 \begin{prop}\label{prop:char}
 Given two submersions $\bs \colon U \to M$ and $\bt \colon U \to N$ with connected fibers, the following are equivalent:
 \begin{enumerate}
  \item The $C^{\infty}(U;\R)$-module $\cF_U = \Gamma_c(U;\ker d\bs) + \Gamma_c(U;\ker d\bt)$ is a foliation;
%  \item The $C^{\infty}(U)$-module $\Gamma(U;\ker d\bs)$ ($\Gamma(U;\ker d\bt)$) is generated by $t$-projectable ($s$-projectable) elements;
  \item There are unique foliations $(M,\cF_M)$ and $(N,\cF_N)$ such that $$s^{-1}(\cF_M)={\bt}^{-1}(\cF_N) = \Gamma_c(U;\ker d\bs) + \Gamma_c(U;\ker d\bt)$$
 \end{enumerate}
\end{prop} 
\begin{note}Notice that the characterization b) is the one used in definition \ref{dfn:bisubmfol}. In this sequel we will focus on characterization a).\end{note}
\begin{proof}
``$a)\Rightarrow b)$:'' This implication follows applying Lemma \ref{lem:submfol} %\mcomment{here we need it only for foliations} 
to the foliation $\cF_U$ and the maps $\bs,\bt$.
``$b)\Rightarrow a)$:'' On every small enough open subset of $U$ there exists a frame for  $ker d\bs$ consisting of vector fields $Y_i$ which are $\bt$-projectable (to elements of $\cF_N$), by the proof of \cite[Prop. 2.11b)]{androulidakis2013a}. In particular  $[\Gamma(U;\ker d\bt),Y_i]\subset \Gamma(U;\ker d\bt)$, so  $[\Gamma_c(U;\ker d\bt),\Gamma_c(U;\ker d\bs)]\subset \cF_U$ and therefore $\cF_U$ is involutive. $\cF_U$ is   locally finitely generated (as $\ker d\bt,\ker d\bs$ are constant rank distributions), hence it is a foliation.

%\mcomment{I have proofs for b) too, should we keep b)? we never use it..}
\end{proof}

Corollary \ref{cor:desingularize}\footnote{This corollary is not used in the rest of this note.} below shows that if a generalised smooth distribution is ``desingularized'' by a pair of submersions, then it is necessarily involutive. 
\begin{cor}\label{cor:desingularize}
 Let $\bs \colon U \to M$ and $\bt \colon U \to N$ be  two surjective submersions   {with connected fibers}, denote ${\cF_U} := \Gamma_c(U;\ker d\bs) + \Gamma_c(U;\ker d\bt)$. If there is a generalised smooth distribution $\cF_M$ of $\vX_c(M)$ such that $\bs^{-1}\cF_M=\cF_U$, then $\cF_M$ is a foliation, and there exists a (unique) foliation $\cF_N$ on $N$ such that $(U,\bs,\bt)$ is a bi-submersion for $\cF_M$ and $\cF_N$.
\end{cor}
\begin{proof}
The existence of the distribution $\cF_M$ such that $\bs^{-1}\cF_M=\cF_U$ implies that elements of $\cF_U$ are $\bs$-projectable vector fields (and they project to $\cF$ by $\bs$). As in the proof of $a)\Rightarrow b)$ in Prop. \ref{prop:char}, we have $[\Gamma_c(U;\ker d\bs),\Gamma_c(U;\ker d\bt)]\subset \cF_U$, which implies that $\cF_U$ is involutive.  $\cF_U$ is also locally finitely generated and the result follows by the implication $a)\Rightarrow b)$ in Prop. \ref{prop:char}.
\end{proof}

%\mcomment{I should understand and write this:}
% In fact, you even gave a counter-example of this situation: You took $U=\R^2$ and $M=N=\R$ and constructed submersions $s=pr_1 : \R^2 \to \R$ and $t(x,y)=y^3$ (or so...) for which none of the above holds.
 
\begin{definition}\label{dfn:bisub1}
 Let $\bs, \bt : U \to M$ be two submersions. We call the triple $(U,\bt,\bs)$ a \textbf{bi-submersion} if it satisfies one of the 2 equivalent conditions in Proposition \ref{prop:char}.
\end{definition}

\begin{exs}\label{exs:bisubmersion}
\begin{enumerate}
\item Let $\cG \gpd M$ be a Lie groupoid with source and target maps $\bs$ and $\bt$ respectively. Then $M$ is endowed with the foliation $\cF = d\bt(\Gamma_c(A\cG))$. It was shown in \cite{AndrSk} that $(\cG,\bt,\bs)$ is a bisubmersion for the foliation $(M,\cF)$. We will reprove this in Example \ref{ex:gpd}.
\item We recall the notion of full dual pair in Poisson geometry from \cite[\S 8]{We}.
Let $M,N$ be Poisson manifolds. A full dual pair consists of a symplectic manifold $(U,\omega)$ with surjective submersions $s\colon U\to M$ and $t\colon U\to N$ which are Poisson and anti-Poisson maps respectively, and such that $ker(s_*|_u)$ and $ker(t_*|_u)$ are symplectic orthogonal subspaces of $T_uU$ for all $u\in U$. In the following we denote by $X_F$ the hamiltonian vector field of the function $F$ ($F$ may be defined on $M$,$N$ or $U$).

A full dual pair is an example of bi-submersion as in Def. \ref{dfn:bisub1}. Indeed, $\Gamma(U;\ker d\bs)$ is generated by $\{X_{t^*g}:g\in C^{\infty}(N)\}$ while $\Gamma(U;\ker d\bt)$ is generated by $\{X_{s^*g}:g\in C^{\infty}(M)\}$, and $[X_{t^*g},X_{s^*f}]=0$ since $ker(s_*)$ and $ker(t_*)$ are symplectic orthogonals. Therefore $\Gamma(U;\ker d\bs)+\Gamma(U;\ker d\bt)$ is involutive.

The foliation induced on $N$ is $\{X_{g}:g\in C^{\infty}(N)\}$, so that its leaves are exactly the sympelctic leaves of $N$, and similarly for $M$. (To check this use that $X_{t^*g}$ is a $t$-projectable vector field on $U$, which $t$-projects to $X_g$).
\end{enumerate}
\end{exs}

In Poisson geometry, it is well known that the underlying symplectic foliations of two Poisson structures which form a full dual pair have equivalent longitudinal and transversal structure. This is true more generally for bi-submersions as in definition \ref{dfn:bisub1}. 

\begin{prop}\label{prop:uniquefol}
Let $s\colon U\to M$ and $t\colon U\to N$ be surjective submersions. Denote $\cF_U:=\Gamma(U;\ker d\bs)+\Gamma(U;\ker d\bt)$.
Let $\cF_M$ be a foliation on $M$ such that $s^{-1}\cF_M=\cF_U$.
Then $(U,s,t)$ is a bi-submersion. In particular, there is a unique foliation  $\cF_N$  on $M$ such that $t^{-1}\cF_N=\cF_U$.
\end{prop}
\begin{proof}
\begin{sloppypar}
By Lemma \ref{lem:submfol} we have $[\Gamma(U;\ker d\bs), \cF_U]\subset \cF_U$. Therefore $[\Gamma(U;\ker d\bs), \Gamma(U;\ker d\bt)]\subset \cF_U$, so $\cF_U$ is involutive. In particular $[\Gamma(U;\ker d\bt), \cF_U]\subset \cF_U$, so applying Lemma \ref{lem:submfol} to the submersion $\bt$ we are done.
\end{sloppypar}
\end{proof}

The following results are proven in \cite[Prop. 2.5]{garmendia2019a}.

\begin{prop}\label{prop:Morita}
Let $(U,\bs,\bt)$ be a bi-submersion. Assume that  $\bs\colon U\to M$ and $\bt\colon U\to N$ are submersions and the $\bs$-fibers and $\bt$-fibers are connected. Then:
\begin{enumerate}
\item There is a bijection between the spaces of leaves of $(N,\cF_N)$ and $(M,\cF_M)$. %the leaf through $x\in N$ corresponds to $s({\bt}^{-1}(x))$.
\item Let $L_x$ be a leaf of $N$ through  a point $x$ and $L_y$ the corresponding leaf of $M$. Choose slices $S_x$ at $x$ and $S_y$ at $y$. Then the foliated manifolds $(S_x,\cF_N|_{S_x})$ and $(S_y,\cF_M|_{S_y})$ are diffeomorphic.
\end{enumerate}
\end{prop}

%Prop. \ref{prop:Morita} is a variation of a well-known statement in Poisson geometry \cite[\S 8]{We}, in which bi-submersions appear naturally as described below.

\subsection{Reformulating the notion of bi-submersion}

%\mcomment{I changed various statements, to keep the logic straight}

 Let $\bs \colon U \to M$ and $\bt \colon U \to N$ be two submersions. We consider the triple $(W,p,q)$, where
 $W := U \times_{\bs} U \times_{\bt} U$ and $p=pr_1, q=pr_3 : W \to U$. Further we denote $\cF_U := \Gamma(U;\ker d\bs) + \Gamma(U;\ker d\bt)$, as in Prop. \ref{prop:char}. The next lemma is an immediate consequence of the definitions.

\begin{lemma}\label{lem:UW} Let $(U,\bt,\bs)$ be a bi-submersion.  Then
  $(W,p,q)$ is a bi-submersion for the foliation $\cF_U$.
\end{lemma}
%\begin{proof}
%\mcomment{missing}
%\end{proof}

\begin{thm}\label{prop:bisub} 
Let $\bs \colon U \to M$ and $\bt \colon U \to N$ be submersions. The following are equivalent:
\begin{enumerate}
\item $(U,\bt,\bs)$ is a bi-submersion. 
\item  For every $\tilde{u}\in U$ there exists a neighborhood $W_{\tilde{u}}$ of $(\tilde{u},\tilde{u},\tilde{u})$ in $W$ and a smooth map $\phi_{\tilde{u}} : W_{\tilde{u}} \to W$ such that $\phi_{\tilde{u}}(u,u,u)=(u,u,u)$ for all $u\in W_{\tilde{u}}$ and $p\circ\phi=q$, $q\circ\phi=p$.
\end{enumerate}
 %For every $u\in U$ there exists a neighborhood $W'$ of $(u,u,u)$ in $W$ and a smooth map $\phi : W' \to W$ such that $\phi(u,u,u)=(u,u,u)$ and $p\circ\phi=q$, $q\circ\phi=p$.
 \end{thm}

The proof of one implication in Prop. \ref{prop:bisub} is easy, and is given just below. The proof of the converse is more involved. We prefer to place this proof in appendix \ref{appendix} so as not to interrupt the flow of ideas.

\begin{proof}[Proof of ``$a)\Rightarrow b)'' $ in Theorem \ref{prop:bisub}]
$(W,p,q)$ is a bi-submersion  for the foliation $\cF_U$ on $U$, by   Lemma \ref{lem:UW}. The embedding of $U$ in $W$, as $\{(u,u,u):u\in U\}$, is a (global) bisection for both the bi-submersion $(W,p,q)$ and the inverse bi-submersion  $(W,q,p)$. In both cases  the diffeomorphism of $U$ it carries is the same, namely, the identity $Id_U$. It follows from item (k) in \S\ref{sec:bisubmgeneral} that for every $\tilde{u}\in U$ there exists a neighborhood $W_{\tilde{u}}$ of $(\tilde{u},\tilde{u},\tilde{u})$ in $W$ and a smooth morphism of bi-submersions  $\phi_{\tilde{u}} : (W_{\tilde{u}},p,q) \to (W,q,p)$ which is the identity on the above bisection. 
\end{proof}
 
%\begin{definition}\label{dfn:bisub2}
% (Provisional) Let $M$ be a smooth manifold. A foliation on $M$ is a triple $(U,t,s)$ where $\bs \colon U \to M$ and $\bt \colon U \to N$ are submersions and there exists a neighborhood $W'$ of $\{(u,u,u): u \in U\}$ in $W$ and a smooth map $\phi : W' \to W$ such that $\phi(u,u,u)=(u,u,u)$ and $p\circ\phi=q$, $q\circ\phi=p$. 
%\end{definition}

Theorem \ref{prop:bisub} allows us to reformulate the notion of bi-submersion as in Definition \ref{dfn:bisub2} below. This definition makes sense beyond the smooth category. 

\begin{definition}\label{dfn:bisub2}
Let $M$ be a smooth manifold. A bi-submersion on $M$ is a triple $(U,\bt,\bs)$ where $\bs \colon U \to M$ and $\bt \colon U \to M$ are submersions and for every $\tilde{u}\in U$ there exists a neighborhood $W_{\tilde{u}}$ of $\{(\tilde{u},\tilde{u},\tilde{u})\}$ in $W$ and a smooth map $\phi_{\tilde{u}} : W_{\tilde{u}} \to W$ such that $\phi_{\tilde{u}}(u,u,u)=(u,u,u)$ for all $u\in W_{\tilde{u}}$ and $p\circ\phi=q$, $q\circ\phi=p$. %\mcomment{Eventually we should let the codomains of $s$ and $t$ be different manifolds.}
\end{definition}
 
%\begin{remark}
%Definition \ref{dfn:bisub2} also makes sense when $U, M, N$ are just topological spaces and $\bs, \bt$ open maps. In view of the results we discussed so far, it may be used to make sense of the notion of the Lie bracket for a topological space. \icomment{Think again about this remark.}
%\end{remark}
 
From the implication $b)\Rightarrow a)$ in Theorem \ref{prop:bisub}, which is proven in appendix \ref{appendix}, we obtain:

\begin{cor} 
A triple $(U,\bt,\bs)$ satisfying definition \ref{dfn:bisub2} is a bi-submersion in the sense of Def. \ref{dfn:bisub1}.
\end{cor}

In \S \ref{sec:examples} we will provide examples for the statements made so far.

\subsection{On the formula of the map $\phi_{\tilde{u}}$}\label{sec:furtherexplanations}

Here we give the explicit formula of the map $\phi_{\tilde{u}}$ at points of the form $(u_1,\tilde{u},u_2)$ such that $u_1$ and $u_2$ are near $\tilde{u}$. We'll see in appendix \ref{appendix} that the general formula of $\phi_{\tilde{u}}$ is similar.

To this end, first note that the following conditions imply condition (b) in Thm. \ref{prop:bisub}:
\begin{enumerate}
\item[(1)] For every $v \in U$ such that $(\tilde{u},\tilde{u},v) \in W_{\tilde{u}}$ we have $\phi_{\tilde{u}}(\tilde{u},\tilde{u},v)=(v,v,\tilde{u})$
\item[(2)] For every $v' \in U$ such that $(v',\tilde{u},\tilde{u}) \in W_{\tilde{u}}$ we have $\phi_{\tilde{u}}(v',\tilde{u},\tilde{u})=(\tilde{u},v',v')$
\end{enumerate}
If we assume that $\phi_{\tilde{u}}$ satisfies (1) and (2), its role becomes much clearer:

Consider $X \in \Gamma(U;\ker d\bs)$ and $Y \in \Gamma(U;\ker d\bt)$ and denote $\varphi^X$ and $\varphi^Y$ their flows. Then for every $t$ sufficiently small the elements $(\varphi_t^X(\tilde{u}),\tilde{u},\tilde{u})$ belong in $W'$, so $\phi_{\tilde{u}}(\varphi_t^X(\tilde{u}),\tilde{u},\tilde{u}) = (\tilde{u},\varphi_t^X(\tilde{u}),\varphi_t^X(\tilde{u}))$. Differentiating this we get $(\phi_{\tilde{u}})_*(X(\tilde{u}),0,0)=(0,X(\tilde{u}),X(\tilde{u}))$. Likewise, for every small $t$ the elements $(\tilde{u},\tilde{u},\varphi^Y_t(\tilde{u}))$ belong in $W'$, so $\phi_{\tilde{u}}(\tilde{u},\tilde{u},\varphi^Y_t(\tilde{u})) = (\varphi^Y_t(\tilde{u}), \varphi^Y_t(\tilde{u}), \tilde{u})$. Again by differentiating we have $(\phi_{\tilde{u}})_*(0,0,Y(\tilde{u}))=(Y(\tilde{u}),Y(\tilde{u}),0)$. Eventually we get: 
\begin{multline*}
(\phi_{\tilde{u}})_*(X(\tilde{u}),0,Y(\tilde{u})) = (\phi_{\tilde{u}})_*(X(\tilde{u}),0,0) + (\phi_{\tilde{u}})_*(0,0,Y(\tilde{u})) \\ = (0,X(\tilde{u}),X(\tilde{u})) + (Y(\tilde{u}),Y(\tilde{u}),0) = (Y(\tilde{u}),X(\tilde{u})+Y(\tilde{u}),X(\tilde{u}))
\end{multline*}
The equation $(\phi_{\tilde{u}})_*(X(\tilde{u}),0,Y(\tilde{u})) = (\phi_{\tilde{u}})_*(X(\tilde{u}),0,0) + (\phi_{\tilde{u}})_*(0,0,Y(\tilde{u}))$ says that for every small $t$ the element $(\varphi^X_t(\tilde{u}),\tilde{u},\varphi^Y_t(\tilde{u}))$ belongs in $W'$. On the other hand, the equation $$(\phi_{\tilde{u}})_*(X(\tilde{u}),0,Y(\tilde{u}))=(Y(\tilde{u}),X(\tilde{u})+Y(\tilde{u}),X(\tilde{u}))$$ gives a way to compute the middle term $z$ in the expression $\phi_{\tilde{u}}(\varphi^X_t(\tilde{u}),\tilde{u},\varphi^Y_t(\tilde{u})) = (\varphi^Y_t(\tilde{u}),z,\varphi^X_t(\tilde{u}))$. Specifically, $$z=\varphi^{X+Y}_{-t}(\tilde{u})$$ %= \varphi^X_{-t}\circ \varphi_{-t}^{\{(\varphi^X_s)^{-1}_* Y \}_{s \in \R}}(\tilde{u})$$ 
For the sake of completeness, we clarify further the flow $\varphi^{X+Y}$. In \cite{Posilicano} we find the following formula for the flow of the sum of two time-dependent vector fields: 
\begin{eqnarray}\label{eqn:timedep1}
\varphi^{X+Y}_{t,s} = \varphi^X_{t,0} \circ \varphi_{t,s}^{(\varphi^X_{0,t})_{\ast}(Y_t)} \circ \varphi^{X}_{0,s} 
\end{eqnarray}
When the vector field $X$ is not time-dependent we have $\varphi^{X}_{t,s}=\varphi^X_{t-s}$, so formula \eqref{eqn:timedep1} becomes:
\begin{eqnarray}\label{eqn:timedep2}
\varphi^{X+Y}_{t-s} = \varphi^X_{t} \circ \varphi_{t-s}^{(\varphi^X_{-t})_{\ast}(Y)} \circ \varphi^{X}_{-s} 
\end{eqnarray}
Whence 
\begin{eqnarray}\label{eqn:timedep3}
z = \varphi^{X+Y}_{-t}(\tilde{u}) = \left(\varphi^X_{-t} \circ \varphi_{-t}^{(\varphi^X_{t})_{\ast}(Y)}\right) (\tilde{u}) 
\end{eqnarray}
Moreover, puting $t=1$ and $s=0$ in \eqref{eqn:timedep2} we obtain: 
\begin{eqnarray}\label{eqn:flowsum}
exp(X+Y)(\tilde{u}) = exp(X)\left(exp\left(\varphi^X_{-1}\right)_{\ast}(Y)(\tilde{u})\right)
\end{eqnarray}

%\begin{remark}\label{rmk:add}
%Consider a path-holonomy bi-submersion $(U \times B^n, \bs,\bt)$ of the foliation $(M,\cF)$ at the point $x \in M$. That is, $U$ is an open neighborhood of $x$ in $M$, $B^n$ is a small ball around $0$ of $\R^n = \cF_x$, $\bs$ is the projection and $\bt$ exponentiation\tobedone The discussion in this section shows that $B^n$ ought to be considered as the additive (abelian) local Lie group. That is to say, the map \tobedone
%\end{remark}

\subsection{Examples}\label{sec:examples}

Here we provide a few examples for Theorem \ref{prop:bisub}. %(and for the its converse, which we know holds, I can write down the proof of the converse). They are also examples for the statements after Def. \ref{dfn:bisub1}.

\begin{ex}\label{ex:gpd}
If $\cG\rightrightarrows M$ is a Lie groupoid, there exists a \emph{global} map $\phi$.
Indeed, put $W = \cG \times_{s} \cG \times_{t} \cG$ and consider the map 
\begin{equation}\label{eq:formula}
\phi_{\cG} : W\to W, \quad (g,h,\zeta) \mapsto (\zeta,gh^{-1}\zeta ,g)
\end{equation}
It obviously satisfies $p\circ\phi_{\cG}=q$, $q\circ\phi_{\cG}=p$, and fixes points $(g,g,g)$ for every $g \in \cG$. Also note that $\phi_{\cG}$ satisfies conditions (1) and (2) in \S \ref{sec:furtherexplanations} for every $(g,g,h)$ and $(h,\zeta,\zeta)$. So locally, the map $\phi_{\cG}$ is given by the formula \eqref{eqn:timedep3}. Further, $\phi_{\cG}$ is an involution, \ie $\phi_{\cG} \circ \phi_{\cG}=Id_W$.
\end{ex}

In fact formula \ref{eq:formula} models the map $\phi_{\tilde{u}}$ mentioned in theorem \ref{prop:bisub} for several interesting bisubmersions, as we show in the next examples.

%\begin{remark}\label{rmk:exgpd}
%The map $\phi_{\cG}$ satisfies conditions (1) and (2) in\S \ref{sec:furtherexplanations}. So locally, the map $\phi_{\cG}$ is given by the formula \eqref{eqn:timedep3}.
%\end{remark}

\begin{ex}\label{ex:add}
Recall from Ex. \ref{ex:pathol} the path-holonomy bi-submersion $(V \times B^n, \bs,\bt)$ of the foliation $(M,\cF)$ at the point $x \in M$. That is, $V$ is an open neighborhood of $x$ in $M$, $B^n$ is a small ball around $0$ of $\R^n = \cF_x$, $\bs$ is the projection and $\bt(x,\lambda_1,\ldots,\lambda_n)=exp\left(\sum_{i=1}^{n}\lambda_i X_i\right)(x)$, where $X_1,\ldots, X_n$ are generators of $\cF$ in a neighborhood of $x$. Now consider $B^n$ as the additive (abelian) local Lie group. That is to say, the map $\phi_{V \times B^n}$ is given by: 
\begin{equation}\label{eq:add}
%\begin{split}
\phi_{V \times B^n}((x,\overrightarrow{\lambda}),(y,\overrightarrow{\mu}),(z,\overrightarrow{\nu}) = 
((z,\overrightarrow{\nu}),(y,\overrightarrow{\lambda} - \overrightarrow{\mu} + \overrightarrow{\nu}),(x,\overrightarrow{\lambda})
%\end{split}
\end{equation}
for every $\overrightarrow{\lambda}, \overrightarrow{\mu}, \overrightarrow{\nu} \in B^n$ such that the sum $\overrightarrow{\lambda} - \overrightarrow{\mu} + \overrightarrow{\nu}$ is defined. Notice that the unique foliation constructed in Prop. \ref{prop:uniquefol} is exactly the original foliation $\cF$ (restricted to $V$). That is because the map $\bt$ is defined using generators of $\cF$.
\end{ex}

\begin{remark}\label{rmk:gpdadd}
Example \ref{ex:add} can also be seen as the local version of example \ref{ex:gpd}. See \cite[Ex. 3.4, item 4]{AndrSk}. 
\end{remark}

\begin{ex}\label{ex:isotropy}
Let $A \to M$ be a Lie algebroid, which is not necessarily integrable. Put $\rho : A \to TM$ its anchor map and fix a point $x \in M$. Let $V$ be an open neighborhood of $x$ in $M$. %The anchor map induces a morphism of $C^{\infty}(V)$-modules $\rho : \Gamma_c(V;A) \to \cF_V$ which is onto. (Here $(M,\cF)$ is the singular foliation induced by $A$, \ie $\cF = \rho(\Gamma_c(A))$.) Since $\Gamma_c(V;A)$ is a free module and $ \cF_V$ is finitely generated, the ideal $\h_V = \ker \rho$ is finitely generated. The fiber $\h_x = \frac{\h_V}{I_x \h_V}$ is a Lie algebra: Indeed, since the anchor map vanishes on sections in $\h_V$, the Leibniz rule for the Lie bracket $[\cdot,\cdot]_A$ of $A$ implies that the formula $[\langle \xi \rangle,\langle \eta \rangle] = \langle [\xi,\eta]_A \rangle$ is well defined for all $\xi,\eta \in \h_V$. This formula defines the Lie bracket of $\h_x$ (our notation $\langle \xi \rangle$ stands for the class in $\h_x$ of $\xi \in \h_V$). 
Let $H_x$ be the connected and simply connected Lie group whose Lie algebra is $\h_x$. Let $U$ be a small neighborhood of $(0,e)$ in $V \times H_x$ and put $\bs : U \to M$ the projection. Consider the submersion $\bt : U \to M$ defined as follows:
\begin{itemize}
\item Choose sections $\xi_1,\ldots,\xi_k$ in $\h_U$ such that $\langle\xi_1\rangle,\ldots,\langle\xi_k\rangle$ is a basis of $\h_x$.
\item Let $(y,h) \in U$. By shrinking $U$ if necessary, we may assume $h$ is sufficiently close to the identity, so that $h=exp(\xi)$ for a unique $\xi \in \h_x$. Let $\xi = \lambda_1\xi_1 + \ldots + \lambda_k\xi_k$ for unique $\lambda_1,\ldots,\lambda_k \in \R$. 
\item Put $\bt(y,h) = exp(\sum_{i=1}^k\lambda_i\rho(\xi_i))(y)$. Of course, $\rho(\xi_i)=0$ for every $i=1,\ldots,k$, so $\bt(y,h)=\bs(y,h)=y$.
\end{itemize}
Now let us show that the triple $(U,\bt,\bs)$ is a bi-submersion. Put $\tilde{u}=(x,e)$. Near the identity element $e \in H_x$ the group multiplication is defined only locally (Baker-Campbell-Hausdorff formula). In other words, there is a small neighbourhood $W$ of $\tilde{u}$ in $U$ where the map $\phi_{\tilde{u}} : W \times_{\bs} W \times_{\bt} W \to W \times_{\bs} W \times_{\bt} W$ with the following formula: $$\phi_{\tilde{u}}((y,h_1),(y,h_2),(y,h_3)) = ((y,h_3),(y,h_1h_2^{-1}h_3),(y,h_1))$$
As we observed in the last item above, this bi-submersion induces the trivial foliation $\cF = 0$ on $V$. That is, the foliation by points of $V$.
\end{ex}

%\begin{remarks}\label{rmk:isotropy}
%\begin{enumerate}
%\item Note that the Lie algebra $\h_x$ is \textit{included} in the usual isotropy algebra of the Lie algebroid $A$. It is the kernel of the linear map $\tilde{\rho}_x : A_x \to \cF_x$ induced by the anchor map at the level of sections, rather than the kernel of the usual anchor map $\rho_x : A_x \to T_x L$ (where $L$ is the leaf at $x$). The vector spaces $\cF_x$ and $T_x L$ are the same if and only if $L$ is a regular leaf. For further details, see \cite[\S 1.3]{AZ3}.
%\item Let $G \gpd M$ be a local Lie groupoid such that $A=AG$. The ideal $\h_V$ corresponds to the singular foliation $(G_V^V,\overrightarrow{\h_V} \cap \overleftarrow{\h_V})$, where $\overrightarrow{\h_V}\cap \overleftarrow{\h_V}$ is the $C^{\infty}(G_V^V)$-module generated by the right-invariant and left-invariant vector fields of $G_V^V$ corresponding to the generators $\xi_1,\ldots,\xi_k$ of $\h_V$. The leaves of this foliation are the isotropy local Lie groups $G_x^x = s^{-1}(x) \cap t^{-1}(x)$. %See \cite[\S 2.1]{AnZa20} for the relation of the path-holonomy bi-submersion defined from the vector fields $\overrightarrow{\xi_1},\ldots,\overrightarrow{\xi_k}$, with the bi-submersion defined in example \ref{ex:isotropy} using the sections $\xi_1,\ldots,\xi_k$.
%\item 
Examples \ref{ex:add} and \ref{ex:isotropy} can be combined to give a bi-submersion for the Weinstein groupoid. This is the subject of \S \ref{sec:pathhol}.
%\end{enumerate}
%\end{remarks}

%\begin{ex}\label{ex:nonint}
%Let $G \gpd M$ be a local Lie groupoid and $U$ an open neighborhood of $x$ in $M$. The anchor map $dt : A_U \to TU$ induces a morphism of $C^{\infty}(U)$-modules $\Gamma_c(U;AG) \to \cF_U$ which is onto. (Here $(M,\cF)$ is the singular foliation induced by $AG$.) Since $\Gamma_c(U;AG)$ is a free module and $ \cF_U$ is finitely generated, the ideal $\h_U = \ker dt$ is finitely generated. Whence it corresponds to a singular foliation $\overrightarrow{\h_U}$ of $G_U^U = s^{-1}(U) \cap t^{-1}(U)$. That is to say, elements of $\overrightarrow{\h_U}$ are the right-invariant vector fields on $G$ corresponding to sections in $\h_U$. So the leaves of $\overrightarrow{\h_U}$ are included in the $s$-fibers of $G$. In fact, since $dt(\overrightarrow{\h_U})=0$, it follows that the leaves are exactly the isotropy local Lie groups $G_x^x = s^{-1}(x) \cap t^{-1}(x)$. 

%Now consider the Lie algebra $\overrightarrow{\h}_{1_x} = \frac{\overrightarrow{\h_U}}{I_{1_x}\overrightarrow{\h_U}}$, where $I_{1_x}$ stands for the functions in $C^{\infty}(G,\R)$ which vanish at $1_x$.  
%\end{ex}

The following examples are by Marco Zambon. They show that the answer to the question ``does every pair of submersions give rise to a bi-submersion?'', is negative.

\begin{ex}\label{ex:marco1}
Fix $f\in C^{\infty}(\RR)$ (a simple case to consider is $f(x)=x^3$).
We consider two submersions $U=\RR^2\to M=\RR$. The first is $s(x,y)=y$. The second is $t(x,y)=f(x)-y$.
Notice that $\ker(s_*)=\pd{x}$ and $\ker(t_*)=\pd{x}+f'\pd{y}$.

$\hat{\cF}$ is involutive if{f} $f'$ is nowhere zero or it is the zero function. This is seen computing $[\pd{x},f'\pd{y}]=f''\pd{y}$, and noticing that $f''$ is a $C^{\infty}(\RR)$-multiple of $f'$ if{f}  $f'$ is nowhere zero or it is the zero function\footnote{We also have that $\hat{\cF}$ is regular exactly when $f'$ is nowhere zero or it is the zero function. Hence, in this very example, $\hat{\cF}$ is involutive if{f} it is regular.}.

The set\footnote{We denote the coordinates of $\R^6$ by $(x_1,y_1,x_2,y_2,x_3,y_3)$.} $W=U \times_s U \times_t U $ is $$\{y_1=y_2,f(x_2)-y_2=f(x_3)-y_3\}\subseteq \R^6$$ Hence it is diffeomorphic to $\RR^4$ with coordinates
$(x_1,y_1,x_2,x_3)$, with projections $p$ mapping to $(x_1,y_1)\in \RR^2$ and $q$ mapping to $(x_3,f(x_3)-f(x_2)+y_1)\in \RR^2$.

We have  $\ker(q_*)=\la \pd{x_1},\pd{x_2}+f'({x_2})\pd{y_1} \ra$ and
$\ker(p_*)=\la \pd{x_2},\pd{x_3} \ra$.
Further, $$p^{-1}(\hat{\cF})=\la \pd{x_1}, \pd{x_1}+f'({x_1})\pd{y_1}, 
\pd{x_2},\pd{x_3}\ra.$$

If $f'$ never vanishes, or if it is the zero function, we see that $\ker(p_*)+\ker(q_*)$ agrees with $p^{-1}(\hat{\cF})$. Otherwise, neither module of vector fields is contained in the other.
% take for example $f=x^2$.

Let us fix a point $(x,y)\in \RR^2$ and try to find a map $\phi:=\phi_{(x,y,x,x)}$ defined near $(x,y,x,x)$. The conditions $p\circ\phi=q$, $q\circ\phi=p$ leave as only possibility maps of the form
$$(x_1,y_1,x_2,x_3)\mapsto (x_3,f(x_3)-f(x_2)+y_1,z,x_1)$$ where $z=z(x_1,y_1,x_2,x_3)$ satisfies $f(z)=f(x_1)+f(x_3)-f(x_2)$. If $f'$ never vanishes then $f$ is a diffeomorphism onto its image, and (nearby $(x,y,x,x)$) we can invert it and take $z=f^{-1}(f(x_1)+f(x_3)-f(x_2))$. 
If $f$ is constant, any smooth function $z$ will do.

In the remaining cases, Thm. \ref{prop:bisub} states that there is a point $(x,y)\in \R^2$ for which there is no smooth function  $z$ as above. For instance, taking $f=x^3$ and the point $(0,y)\in \RR^2$ (for any $y$), the only possibility for $z$ is $z(x_1,y_1,x_2,x_3)=\sqrt[3]{x_1^3+x_3^3-x_2^3}$, which is continuous but not smooth.
\end{ex}

We now present an example where  $\hat{\cF}$ is regular but not involutive. 
\begin{ex}\label{ex:marco2}
We consider two submersions $U=\RR^3\to M=\RR$. The submersion $s$ maps $(x,y,z)$ to $(x,z)$, while $t$ maps it to $(y,z-xy)$.
Notice that $\ker(s_*)=\la\pd{y}\ra$ and $\ker(t_*)=\la\pd{x}+y\pd{z}\ra$.
So $\hat{\cF}$ is non-involutive (it is the kernel of the standard contact 1-form $dz-ydx$).

The set\footnote{We denote the coordinates of $\R^9$ by $(x_1,y_1,z_1,x_2,y_2,z_2,x_3,y_3,z_3)$} $W=U \times_s U \times_t U $ is the subset of $\RR^9$ given by $\{x_1=x_2,z_1=z_2,y_2=y_3,z_2-x_2y_2=z_3-x_3y_3\}$.
 Hence it is diffeomorphic to $\RR^5$ with coordinates
$(x_1,y_1,z_1,y_2,x_3)$, with projections $p$ mapping to $(x_1,y_1,z_1)\in \RR^3$ and $q$ mapping to $(x_3,y_2,z_1+(x_3-x_1)y_2)$.

We have $\ker(p_*)=\la \pd{y_2},\pd{x_3} \ra$ and
 $\ker(q_*)=\la \pd{y_1},\pd{x_1}+ y_2 \pd{z_1} \ra$.
Further, $$p^{-1}(\hat{\cF})=\la \pd{y_1}, \pd{x_1}+ y_1 \pd{z_1}, 
\pd{y_2},\pd{x_3}\ra.$$
So neither of  $\ker(p_*)+\ker(q_*)$  or  $p^{-1}(\hat{\cF})$ is contained in the other.

We fix $(x,y,z)\in \RR^3$ and try to find a map $\phi:=\phi_{(x,y,z,y,x)}$ defined near $(x,y,z,y,x)$. The condition $q\circ\phi=p$ leaves as only possibility 
$$(x_1,y_1,z_1,y_2,x_3)\mapsto (x_3,y_2,z_1+(x_3-x_1)y_2,z,w).$$ Here
the condition $p\circ\phi=q$ forces $z=y_1$ and $w=x_1$. But
this leads to contradiction, as $z_1+(x_3-x_1)y_2+(x_1-x_3)y_1\neq z_1$.
We conclude that a map $\phi$ does not exist.  
\end{ex}
%\mcomment{maybe also ex $U=M\times N$ o $s=t\colon U\to M$.}

\section{Bi-submersions of an integrable Lie algebroid}\label{sec:bisubmintegrable}

The purpose of this section is to extend the formulation of bi-submersions we gave in definition \ref{dfn:bisub2} to the bi-submersions of an integrable Lie algebroid introduced in \cite[Definition 2.4]{Za18}. We start by recalling this definition%\footnote{A detailed exposition is given in \cite[\S 2.2]{Za18}.} 
. To this end, we have to fix some notation first. 

\begin{notation}\label{not:relbisubm}
Let $\cG \gpd M$ be a Lie groupoid. We denote its source and target maps by $\bs, \bt$ respectively and its Lie algebroid by $A\cG$. Recall that $A\cG = (\ker d\bs)|_M$ is the pullback of the vertical vector bundle $\ker d\bs \to \cG$ by the unit inclusion map $M \to \cG$. 
\begin{enumerate}
\item Put $\Gamma^{RI}_c T\cG$ the $C^{\infty}(\cG)$-module of (compactly supported) right-invariant vector fields of $\cG$. Recall that this is a submodule of $\Gamma_c \ker d\bs$. We denote its elements by $\rar{\sigma}$. Recall that every section $\sigma$ of the Lie algebroid $A\cG$ corresponds to a unique right-invariant vector field $\rar{\sigma}$.
\item Likewise, put $\Gamma^{LI}_c T\cG$ the module of (compactly supported) left-invariant vector fields of $\cG$, which is a submodule of $\Gamma_c \ker d\bt$. We denote a vector field as such by $\lar{\sigma}$. It also corresponds to a unique section of $A\cG$ as well.
\item We denote $(M,\cF^{\cG})$ the foliation defined by $\cG$, that is $\cF^{\cG} = d\bt(A\cG)$.
\end{enumerate}
\end{notation}

Definition \ref{dfn:relbisubm} below is really \cite[Definition 2.4]{Za18}, applied to the singular subalgebroid $\Gamma_c(A\cG)$.

\begin{definition}\label{dfn:relbisubm}
A \textit{bi-submersion of $A\cG$} is a triple $(U,\psi,\cG)$ such that:
\begin{enumerate}
\item $\psi : U \to \cG$ is a smooth map.
\item The maps $\bs_U = \bs \circ \psi$ and $\bt_U = \bt \circ \psi$ are submersions.
\item For every $\sigma \in \Gamma_c(A\cG)$ there is $Z \in \cX_c(U)$ which is $\psi$-related to $\rar{\sigma}$ and $W \in \cX_c(U)$ which is $\psi$-related to $\lar{\sigma}$.
\item $\psi^{-1}(\Gamma^{RI}_c T\cG) = \Gamma_c(U;\ker d\bs_U)$ and $\psi^{-1}(\Gamma^{LI}_c T\cG) = \Gamma_c(U;\ker d\bt_U)$.
\end{enumerate}
\end{definition}

\begin{exs}\label{exs:relbisubm}
\begin{enumerate}
\item The Lie groupoid $\cG \gpd M$ is a bi-submersion of $A\cG$ with $\psi=\id$. Notice that $(\cG,\bs,\bt)$ is also a bisubmersion of the foliation $(M,\cF^{\cG})$.
\item Suppose the $\bs$-fibers of $\cG \gpd M$ have dimension $m$ and let $V$ be an open subset of $\cG$ which contains $1(\bs(V))$ (here $1 : M \to \cG$ is the unit embedding). There is an open subset $U \subseteq M \times \R^m$ which contains $\bs(V) \times \{0\}$ and a diffeomorphism $\psi : U \to V$\footnote{Choose a frame $\sigma_1,\ldots,\sigma_m$ of $\Gamma(\bs(U);A\cG)$ and consider the corresponding right-invariant vector fields $\rar{\sigma}_1,\ldots,\rar{\sigma}_m$. Define $\psi(x,\lambda_1,\ldots,\lambda_m)=exp(\sum_{i=1}^{m}\lambda_i\rar{\sigma}_i)(1_x)$.}. The triple $(U,\psi,\cG)$ is a bi-submersion of $A\cG$.
\item Let $(U,\bs,\bt)$ be a bi-submersion of the foliation $(M,\cF)$. Consider the pair groupoid $M \times M$ and the map $\psi = (\bt,\bs)$. It is shown in \cite[\S 2.3.1]{Za18} that $(U,\psi,M \times M)$ satisfies items (a) and (c) of definition \ref{dfn:relbisubm}. Item (b) is satisfied only for vector fields $X \in \cF$, not for every vector field of $M$. In fact, $(U,\psi,M \times M)$ is a bi-submersion for the singular subalgebroid $\cF$ of $TM=A(M \times M)$, in the sense of \cite[Definition 2.4]{Za18}.
\end{enumerate}
\end{exs}

Looking at the third item in examples \ref{exs:relbisubm} and following the definitions, we see that for every $\tilde{u} \in U$, the following diagram commutes:
%\begin{displaymath}
% \xymatrix{
% u \ar[dr] & & y_{\lambda}x_{\lambda}(u) \ar@{.>}[dr] & & y_{-\lambda}x_{-\lambda}y_{\lambda}x_{\lambda}(u) \\
%           & x_{\lambda}(u) \ar[dr] \ar@{.>}[ur] & & x_{-\lambda}y_{\lambda}x_{\lambda}(u) \ar[ur] &  \\
%           & &  z(u,\lambda) \ar[ur] & & 
% }  
\begin{displaymath}
% \xymatrix@R+1pt@C+2pc{ %%% R changes length of vertical arrows and C changes length of horizontal arrows. To center the label, add a minus sign to the \ar[?] command, for instance: \ar[r]^-{\psi \times_{\bs}\psi \times_{\bt}\psi}
 \xymatrixcolsep{4pc}\xymatrix{
 W_{\tilde{u}} \ar[r]^-{\psi \times_{\bs}\psi \times_{\bt}\psi} \ar[d]_{\phi_{\tilde{u}}} & (M \times M) \times_{pr_2} (M \times M) \times_{pr_1} (M \times M) \ar[d]^{\phi_{M \times M}} \\
 W_{\tilde{u}} \ar[r]^-{\psi \times_{\bs}\psi \times_{\bt}\psi} & (M \times M) \times_{pr_2} (M \times M) \times_{pr_1} (M \times M) 
 }  
\end{displaymath}
Note that the formula of the map $\phi_{M \times M}$ is the one defined in Example \ref{ex:gpd}, in the particular the case of the pair groupoid $M \times M \gpd M$. In the second item of examples \ref{exs:relbisubm}, since the map $\psi$ is a diffeomorphism, it is straightforward that there is a map $\phi_{\tilde{u}}$ such that the following diagram commutates:
\begin{equation}\label{diag:relbisubm}
% \xymatrix@R+1pt@C+2pc{ %%% R changes length of vertical arrows and C changes length of horizontal arrows. To center the label, add a minus sign to the \ar[?] command, for instance: \ar[r]^-{\psi \times_{\bs}\psi \times_{\bt}\psi}
 \xymatrixcolsep{4pc}\xymatrix{
 W_{\tilde{u}} \ar[r]^-{\psi \times_{\bs}\psi \times_{\bt}\psi} \ar[d]_{\phi_{\tilde{u}}} & \cG \times_{\bs} \cG \times_{\bt} \cG \ar[d]^{\phi_{\cG}} \\
 W_{\tilde{u}} \ar[r]^-{\psi \times_{\bs}\psi \times_{\bt}\psi} & \cG \times_{\bs} \cG \times_{\bt} \cG 
 }  
\end{equation}

As we saw in \S \ref{sec:furtherexplanations}, the explicit formula of $\phi_{\tilde{u}}$ in diagram \ref{diag:relbisubm} is
\begin{eqnarray}\label{eqn:formula}
\phi_{\tilde{u}}(u_1,u,u_2) = (u_2,\varphi^{X+Y}_{-t}(u),u_1)
\end{eqnarray}

%In Lemma \ref{lem:formula} we give the explicit formula of the map $\phi_{\tilde{u}}$ in diagram \ref{diag:relbisubm}.

%\begin{lemma}\label{lem:formula}
%Let $(u_1,u,u_2) \in W_{\tilde{u}}$ such that $u_1 = \varphi^X_t(u)$ and $u_2=\varphi^Y_t(u)$ for some $X \in \Gamma_c(U;\ker d\bs_U)$, $Y \in \Gamma_c(U;\ker d\bt_U)$ and $t \in \R$. Then 
%\begin{eqnarray}\label{eqn:formula}
%\phi_{\tilde{u}}(u_1,u,u_2) = (u_2,\varphi^{X+Y}_{-t}(u),u_1)
%\end{eqnarray}
%\end{lemma}
%\begin{proof}
%We use the results in \S \ref{sec:furtherexplanations}. Notice that the map $\phi_{\cG}$ satisfies $\phi_{\cG}(g,g,\zeta) = (\zeta,\zeta,g)$ and $\phi_{\cG}(g,\zeta,\zeta) = (\zeta, g, g)$. Since $\psi$ is a diffeomorphism, it follows that $\phi_{\tilde{u}}$ satisfies the same conditions. The result follows. %Therefore, if $(u_1, u, u_2)$ is an element of $W_{\tilde{u}}$ such that $u_1 = \varphi^X_t(u)$ and $u_2=\varphi^Y_t(u)$ for some $X \in \Gamma_c(U;\ker d\bs_U)$, $Y \in \Gamma_c(U;\ker d\bt_U)$ and $t \in \R$, we have: $$\phi_{\tilde{u}}(u_1,u,u_2) = (u_2,\varphi^{X+Y}_{-t}(u),u_1)$$
%\end{proof}

\subsection{Reformulating the notion of bi-submersion for $A\cG$}\label{sec:reform}

Recall \cite[Lemma 2.12]{Za18} that if $(U,\psi,\cG)$ is a bi-submersion of $A\cG$ in the sense of Dfn. \ref{dfn:relbisubm}, then $(U,\bs_{\psi},\bt_{\psi})$ is a bi-submersion of the underlying foliation $(M,\cF^{\cG})$, where $\bs_{\psi} = \bs \circ \psi$ and $\bt_{\psi}=\bt \circ \psi$. The purpose of this section is to show the converse of this statement and compare with the notion of bi-submersion introduced in Definition \ref{dfn:bisub2}. This is essential in order to generalise the notion of bi-submersion to non-integrable Lie algebroids in \S \ref{sec:pathhol}. %More precisely, we prove the following: 
%\begin{itemize}
%\item If $(U,\bt_U,\bs_U)$ is a bi-submersion of the foliation $(M,\cF^\cG)$ and $\psi : U \to \cG$ is a morphism of bi-submersions then $(U,\psi,\cG)$ is a bi-submersion of $A\cG$.
%\item In this case, for every $\tilde{u} \in U$ the map $\phi_{\tilde{u}}$ makes diagram \eqref{diag:relbisubm} commute.
%\end{itemize}

\begin{prop}\label{prop:bisubmfolgpd1}
Let $(U,\bt_U,\bs_U)$ be a bi-submersion of the foliation $(M,\cF^\cG)$ induced by the Lie groupoid $\cG \gpd M$. If $\psi : U \to \cG$ is a smooth map which is a submersion and satisfies $\bs_U = \bs \circ \psi$ and $\bt_U = \bt \circ \psi$, then $(U,\psi,\cG)$ is a bi-submersion of $A\cG$, in the sense of Dfn. \ref{dfn:relbisubm}.
\end{prop}
\begin{proof}
We only need to confirm that items (c) and (d) in Dfn. \ref{dfn:relbisubm} hold. Item (c) holds because $\psi$ is a submersion. For item (d), we will only prove that $\psi^{-1}(\Gamma^{RI}_c T\cG) = \Gamma_c(U;\ker d\bs_U)$. The proof of the equality $\psi^{-1}(\Gamma^{LI}_c T\cG) = \Gamma_c(U;\ker d\bt_U)$ is similar.

%The inclusion $\psi^{-1}(\Gamma^{RI}_c T\cG) \subseteq \Gamma_c(U;\ker d\bs_U)$ follows from $\bs_U = \bs \circ \psi$ and the fact that the right-invariant vector fields of $\cG$ are $\bs$-vertical, namely $\Gamma^{RI}_c(T\cG) \subseteq \Gamma_c(\cG;\ker d\bs)$.

%\underline{Claim: $\Gamma_c(U;\ker d\bs_U) \subseteq \psi^{-1}(\Gamma_c^{RI}(T\cG))$.} 
The foliation $(M,\cF^\cG)$ is a singular subalgebroid of the Lie algebroid $TM = A(M \times M)$, in the sense of \cite{Za18}. Consider the $C^{\infty}(M \times M)$-module $\overrightarrow{\cF^{\cG}}$ of right-invariant vector fields of the pair groupoid $M \times M$ corresponding to $\cF^{\cG}$ (\cf item (a) in Notation \ref{not:relbisubm} and \cite[\S 1.4]{Za18}). It was shown in \cite[Prop. 2.10]{Za18} that $(U,\bs_U,\bt_U)$ is a bisubmersion of $(M,\cF^{\cG})$ if and only if $(U,(\bs_U,\bt_U),M \times M)$ is a bi-submersion in the sense of definition \cite[Definition 2.4]{Za18}. On the other hand, since $\cF^{\cG} = d\bt(\Gamma_c A\cG)$, the correspondence of sections of $A\cG$ with right-invariant vector fields of $\cG$ implies $\overrightarrow{\cF^{\cG}}=(d\bt,d\bs)(\Gamma_c^{RI}T\cG)$. Therefore: $$\psi^{-1}(\Gamma_{c}^{RI}T\cG)=\psi^{-1}((\bs,\bt)^{-1}(\overrightarrow{\cF^{\cG}}))=(\bs_U,\bt_U)^{-1}(\overrightarrow{\cF^{\cG}})=\Gamma_c(U;\ker d\bs_U)$$
\end{proof}

We focus on bi-submersions as the ones discussed in the second example of Ex. \ref{exs:relbisubm}. Let us define them precisely:
\begin{itemize}
\item Let $U \subset M$ an open neighbourhood of $x \in M$ and $\sigma = \{ \sigma_1,\ldots,\sigma_n\}$ a local frame of sections which commute by the Lie bracket. Here $\sigma_i \in \Gamma_{c}(U;A\cG)$  for every $i = 1,\ldots,n$.
\item Put $\rar{\sigma}_1,\ldots,\rar{\sigma}_n \in \Gamma_c^{RI}T\cG$ and $\lar{\sigma}_1,\ldots,\lar{\sigma}_n \in \Gamma_c^{LI}(T\cG)$ the right-invariant and left-invariant vector fields corresponding to the frame $\sigma$. For every $i=1,\ldots,n$ put $X_i = d\bt_{\cG}(\sigma_i)$ the associated vector field in $(M,\cF)$.
\item Since the map $d_{x}\bt_{\cG} : A_x \cG \to \cF_{x}$ is surjective, the restrictions to $U$ of the vector fields $X_1,\ldots,X_n$ generate $\cF|_{U}$. Put $(U^{\sigma},\bs_U^{\sigma},\bt_U^{\sigma})$ the path-holonomy bi-submersion defined by $X_1,\ldots,X_n$. Notice that $U^{\sigma} \subseteq M \times \R^n$, where $n$ is the rank of $A\cG$.
\item The diffeomorphism $\psi^{\sigma} : U^{\sigma} \to \cG$ defined by $$\psi^{\sigma}(y,\lambda_1,\ldots,\lambda_n) = exp(\sum_{i=1}^n\lambda_i \rar{\sigma}_i)(1_y)$$ makes $(U^{\sigma},\psi^{\sigma},\cG)$ a bi-submersion of $A\cG$ in the sense of Dfn. \ref{dfn:relbisubm}.
\item For every $\tilde{u} \in U^{\sigma}$ put $W^{\sigma}_{\tilde{u}} \subseteq U^{\sigma} \times_{\bs_U} U^{\sigma} \times_{\bt_U} U^{\sigma}$ an appropriate neighbourhood of $(\tilde{u},\tilde{u},\tilde{u})$. Then there is a map $\phi^{\sigma}_{\tilde{u}}$ which makes the following diagram commute:
\begin{eqnarray}\label{diag:constrank}
\xymatrixcolsep{4pc}\xymatrix{
 W^{\sigma}_{\tilde{u}} \ar[r]^-{\psi^{\sigma} \times_{\bs}\psi^{\sigma} \times_{\bt}\psi^{\sigma}} \ar[d]_{\phi^{\sigma}_{\tilde{u}}} & \cG \times_{\bs} \cG \times_{\bt} \cG \ar[d]^{\phi_{\cG}} \\
 W^{\sigma}_{\tilde{u}} \ar[r]^-{\psi^{\sigma} \times_{\bs}\psi^{\sigma} \times_{\bt}\psi^{\sigma}} & \cG \times_{\bs} \cG \times_{\bt} \cG 
} 
\end{eqnarray}
The explicit formula of $\phi^{\sigma}_{\tilde{u}}$ is \eqref{eqn:formula}.
\end{itemize}

\begin{definition}\label{dfn:bisubmalgdrank}
The bisubmersion $(U^{\sigma},\psi^{\sigma},\cG)$ of $A\cG$ is called the \textit{path-holonomy bi-submersion of $A\cG$ arising from the local frame $\sigma$}. 
\end{definition}
Notice that path-holonomy bi-submersions arising from local frames of $A\cG$ are also path-holonomy bi-submersions of the foliation $(M,\cF)$ defined by the Lie algebroid $A\cG$. The map $\psi$ is defined by exponentiating the right-invariant vector fields associated with the frame $\sigma$. So in this construction, the integrability of $A\cG$ is crucial in the construction of bi-submersions as such.

Also notice that the commutativity of diagram \ref{diag:constrank} follows automatically from the explicit formulas of the maps involved. Diagram \ref{diag:constrank} binds together the definition of bi-submersion \ref{dfn:relbisubm} and its refomulation in Proposition \ref{prop:bisubmfolgpd1}, with the algebraic reformulation of bi-submersion we gave in Theorem \ref{prop:bisub}.  %fact that the map $\psi$ commutes with the source and target maps. In other words, it follows because $\psi$ is a morphism of bi-submersions for the foliation $(M,\cF)$. 

%In the next section we will show that path-holonomy bisubmersions arising from local frames exist for a non-integrable Lie algebroid $\rho : A \to TM$ as well. That is because, as pointed out in \cite{CrFeLie}, given a local section of $A$ it is possible to integrate it in the (topological) monodromy groupoid associated with $A$. 

%Starting from a local frame $\sigma = \{\sigma_1,\ldots,\sigma_n\}$ of $A$ we obtain vector fields $X_i = \rho(\sigma_i)$ which generate locally the foliation $(M,\cF)$ underlying the algebroid. These give rise to the bi-submersion $(U^{\sigma},\bs_{U^{\sigma}},\bt_{U^{\sigma}})$ of 

%\begin{prop}\label{prop:bisubmfolgpd2}
%Let $(U,\bs_U,\bt_U)$ a bi-submersion of the foliation $(M,\cF^{\cG})$ and $\psi : U \to \cG$ as in Prop. \ref{prop:bisubmfolgpd}. For every $\tilde{u} \in U$ there is $\phi_{\tilde{u}} : W_{\tilde{u}} \to W_{\tilde{u}}$ such that diagram \eqref{diag:relbisubm} commutes.
%\end{prop}
%\begin{proof}

%\end{proof}

%\section{Bi-submersions of a non-integrable Lie algebroid}\label{sec:nonsmoothbisubm}

\section{The path-holonomy bi-submersion of a non-integrable Lie algebroid}\label{sec:pathhol}

In \S \ref{sec:dfnWA} we give the definition of a bi-submersion for a Lie algebroid and in \ref{sec:constrWA}, we give the explicit construction of a bi-submersion as such. The existence of bi-submersions as such and the construction in \cite[Appendix A]{Za18} show that we can apply the $C^{\ast}$-functor to any Lie algebroid, regardless its integrability.

\subsection{Definition of bi-submersion for a Lie algebroid}\label{sec:dfnWA}

Let us make a fresh start. Let $A$ be a Lie algebroid over $M$, which is not necessarily integrable. Put $(M,\cF)$ the singular foliation it induces. Let $W(A) \gpd M$ be the Weinstein groupoid of $A$. Put $\bs, \bt : W(A) \to M$ its source and target maps.

\begin{definition}\label{dfn:bisubmnonint}
A bi-submersion of $A$ is a pair $((U,\bs,\bt),\psi)$, where:
\begin{enumerate}
\item $(U,\bs_U,\bt_U)$ is a bi-submersion of the foliation $(M,\cF)$ in the sense of \ref{dfn:bisub2}.
\item $\psi : U \to W(A)$ is an open map such that $\bs\circ\psi=\psi_U$ and $\bt\circ\psi=\bt_U$.
\item For every $u \in U$ the following diagram commutes:
\begin{eqnarray}\label{diag:weinstein}
\xymatrixcolsep{4pc}\xymatrix{
 W_{u} \ar[r]^-{\psi \times_{\bs}\psi \times_{\bt}\psi} \ar[d]_{\phi_{u}} & W(A) \times_{\bs} W(A) \times_{\bt} W(A) \ar[d]^{\phi_{W(A)}} \\
 W_{u} \ar[r]^-{\psi \times_{\bs}\psi \times_{\bt}\psi} & W(A) \times_{\bs} W(A) \times_{\bt} W(A) 
} 
\end{eqnarray}
\end{enumerate}
\end{definition}

\subsection{Construction of path-holonomy bisubmersions}\label{sec:constrWA}

Now we give an explicit construction of a bi-submersion as in Dfn. \ref{dfn:bisubmnonint}, arising from the choice of a local frame for the non-integrable Lie algebroid $A$. %That is because, as pointed out in \cite{CrFeLie}, given a local section of $A$ it is possible to integrate it in the (topological) monodromy groupoid associated with $A$.

Let us fix a point $x \in M$ and an open neighbourhood $V$ of $x$ in $M$. Put $L$ the leaf of the algebroid $A$ at $x$ and $\cF_x$, $\h_x$ the fibers of $\cF_V=\rho(\Gamma_c(V,A))$ and $\h_V$ respectively. Choose a local frame $\mu_1,\ldots \mu_n, \nu_1,\ldots \nu_k$ of the vector bundle $A_V$ such that $\{\xi_i = \rho(\mu_i) : 1\leq i \leq n\}$ is a minimal set of generators of $\cF_V$ and $\{\nu_i : 1\leq i \leq k\}$ a minimal set of generators of $\h_V$. In other words, the classes $[\mu_1],\ldots,[\mu_n]$ are a basis of $\cF_x$ and the classes $[\nu_1],\ldots,[\nu_k]$ are a basis of $\h_x$. We also choose a smooth fiberwise inner product of the vector bundle $A$, to make possible the constructions in \S \ref{sec:Apathspl}.

Recall the bi-submersions constructed in examples \ref{ex:add} and \ref{ex:isotropy}:
\begin{itemize}
\item $(V \times B^n,\bt,\bs)$ where $B^n$ is a small ball in $\cF_x$ whose center is zero, $\bs$ is the projection and $\bt$ is the exponential defined using $\xi_1,\ldots,\xi_n$.
\item $(U,t,s)$ where $U$ is a small neighborhood of $(0,e)$ in $V \times H_x$, where $H_x$ is the connected and simply connected Lie group which integrates $\h_x$. Also $s=t$ is the projection.
\end{itemize}

Now consider the fibered product $$Z = (V \times B^n) \times_{\bt,s} U = V \times B^n \times \tilde{H}_x$$ where $\tilde{H}_x$ is a small neighbourhood of $H_x$ at the identity. Put  $s_Z : Z \to M$ the projection and define $t_Z : Z \to M$ by $t(y,\overrightarrow{\lambda},g) = \bt(y,\overrightarrow{\lambda})$ for every $\overrightarrow{\lambda} = (\lambda_1,\ldots,\lambda_n) \in B^n$. Both of these maps are submersions.

\begin{prop}\label{prop:bisubmWeinstein1}
The triple $(Z,t_Z,s_Z)$ is a bi-submersion. The underlying foliation is $\cF_V$.
\end{prop}
\begin{proof}
Fix the point $\tilde{u} = (x,0,e)$ in $Z$ and consider (up to shrinking $Z$ if necessary) the map $\phi_{Z} : Z \times_{s_Z} Z \times_{t_Z} Z \to Z \times_{s_Z} Z \times_{t_Z} Z$ defined by: 
\begin{equation*}
\begin{split}
& \phi_{Z}((y_1,\overrightarrow{\lambda}_1,g_1),(y_2,\overrightarrow{\lambda}_2,g_2),(y_3,\overrightarrow{\lambda}_3,g_3)) = \\  & ((y_3,\overrightarrow{\lambda}_3,g_3),(y_2,(\overrightarrow{\lambda}_1 - \overrightarrow{\lambda}_2 + \overrightarrow{\lambda}_3),g_1 g_2^{-1} g_3),(y_1,\overrightarrow{\lambda}_1,g_1))
\end{split}
\end{equation*}
It is straightforward that this map satisfies theorem \ref{prop:bisub}.
\end{proof}

\begin{remark}
If $x \in M_A^c$ we have $\h_x = 0$, in other words $k=0$ in the above construction. In this case, $(Z,t_Z,s_Z)$ is just the path-holonomy bisubmersion $(V \times B^n,\bt,\bs)$. This and Prop. \ref{prop:bisubmWeinstein2} justify item (j) in \S \ref{sec:extensions}.
\end{remark}

\begin{remark}
Let $(M,\pi)$ be a Poisson manifold and consider the Lie algebroid structure on the cotangent bundle $T^* M$. Recall \cite[Prop. 2.1]{AZ3} that the Lie algebra $\h_x$ is abelian in this case, so the group element $g_1 g_2^{-1} g_3$ is understood in the additive sense.
\end{remark}

\begin{prop}\label{prop:bisubmWeinstein2}
There is a map $\psi : Z \to W(A)$ which makes diagram \ref{diag:relbisubm} commute. (Replacing $\cG$ by $W(A)$, $U$ with $Z$ and $\phi_{\tilde{u}}$ with $\phi_Z$ in that diagram.)
\end{prop}
\begin{proof}
Let $\g$ be a finite-dimensional Lie algebra. Recall \cite{DuistKolk} that the connected and simply connected Lie group $G$ which integrates $\g$ is the quotient $P_0(\g)/\sim$ of the Banach manifold $P_0(\g)$ of paths in $\g$ by the homotopy relation defined in \cite[Prop. 1.13.4]{DuistKolk}. Whence, for any element $(y,\overrightarrow{\lambda},g)$ in $Z$ there are lifts $\alpha_{\overrightarrow{\lambda}} \in P_0(\R^n)$ and $\alpha_{g} \in P_0(\h_x)$ of $\overrightarrow{\lambda}$ and $g$ respectively. Because of Prop. \ref{prop:Apath2}, every element $(y,\alpha_{\overrightarrow{\lambda}} \oplus \alpha_{g})$ defines an $A$-path in the Lie algebroid $A^{\cF}_{L \cap V} \oplus \h_{L \cap V}$, whence an $A$-path in $A_{L \cap V}$. Put $\psi(y,\overrightarrow{\lambda},g) \in W(A)$ the $A$-homotopy class of this path. The map $\psi$ is well defined because of Remark \ref{rmk:Apath}. The commutativity of diagram \ref{diag:relbisubm} follows from the construction.
\end{proof}

Propositions \ref{prop:bisubmWeinstein1} and \ref{prop:bisubmWeinstein2} justify the next definition.

\begin{definition}\label{dfn:bisubmWeinstein}
Let $(Z,s_Z,t_Z)$ the triple constructed in Prop. \ref{prop:bisubmWeinstein1}. The map $\psi : Z \to W(A)$ is called the minimal path-holonomy bi-submersion of the Weinstein groupoid $W(A)$ at the point $x$ in $M$.
\end{definition}

\section{Questions to be explored in the future}

There existence of path-holonomy bi-submersions for the Weinstein groupoid opens the way for a much more thorough treatment of non-integrable Lie algebroids. That is, the development of both Lie theory and Noncommutative Geometry for algebroids as such. Here are a few of the results we intend to give in the future:

%\begin{remarks}
\begin{enumerate}
\item It is well known that every Lie algebroid $A$ integrates to a local Lie groupoid $\cG(A) \to M$. The map $\psi : Z \to W(A)$ we constructed in \S \ref{sec:constrWA} can probably be extended to a map $\cG(A) \to W(A)$, which ``differentiates'' to $\id : A \to A$. This differentiation process is explained in \cite{AnZa20}.
\item The construction in \cite[Appendix A]{Za18} shows that, using the path-holonomy bi-submersions in Definition \ref{dfn:bisubmWeinstein}, we can attach $C^*$-algebras (full and reduced) to a Lie algebroid $A$. Thanks to items (i), (j) and (k) in \S \ref{sec:extensions} as well as \cite[Notation 1.2]{AS4}, the \textit{full} $C^{\ast}$-algebra of $A$ is an exact sequence: $$0 \to C^{\ast}(A_{M_A}) \to C^{\ast}(A) \to C^{\ast}(A_{M_A^c}) \to 0$$ The $K$-theory of $C^{\ast}(A_{M_A^c})$ is computable because the restriction of $W(A)$ to $M_A^c$ is a Lie groupoid. 
\item Also, the decomposition $M = M_A \cup M_A^c$ makes possible the application of the blow-up constructions in \cite{Mohsen2021}. This paves the way for the development of the appropriate pseudodifferential calculus and index theory for non-integrable Lie algebroids. We intend to carry out this program in forthcoming papers.
\end{enumerate}
%\end{remarks}

\appendix

\section{The proof of Theorem \ref{prop:bisub}}\label{appendix}

Recall that we just need to prove the converse direction of Theorem \ref{prop:bisub}.

As being a bi-submersion is a local condition, we
fix $\tilde{u}\in U$ and work in a neighborhood of it.
 Fix a basis of sections $\{X_1,\ldots,X_n\}$ of $\Gamma(U;\ker d\bs)$ defined in a neighbourhood $\hat{U}$ of  $\tilde{u}$, so that all $X_i$ commute with each other.
For any $u\in \tilde{U}$ and $\xi\in \RR^n$ small enough put $\alpha_{\xi}(u)=exp(\sum_{i=1}^n \xi_i X_i)(u)$. 
%Then $B_n\to U, \xi \mapsto \alpha_{\xi}(\tilde{u})$ is a diffeomorphism onto its image.

Likewise, fix a basis of pairwise commuting sections $\{Y_1,\ldots,Y_n\}$ of $\Gamma(U;\ker d\bt)$ on $\hat{U}$. Define  $\beta_{\eta}(u)=exp(\sum_{i=1}^n \eta_i Y_i)(u)$. We assume that  there is 
a non-empty open neighborhood $\tilde{U} $ of $\tilde{u}$ in $U$ 
such that 
$\alpha_{\xi}(u)$ and 
$\beta_{\eta}(u)$ are well-defined for all $\xi,\eta\in B_n$ (this can be always arranged rescaling suitably the $X_i$ and $Y_i$).
% and obtain $B_n\to U, \eta \mapsto \beta_{\eta}(\tilde{u})$, a diffeomorphism onto its image.

This defines a map
\begin{align*}
\chi : B_n \times 
\tilde{U} \times B_n &\to U \times_s U \times_t U,\\
(\xi,u,\eta)&\mapsto (\alpha_{\xi}(u),u, \beta_{\eta}(u)),
\end{align*}
which is a diffeomorphism onto a neighborhood of $(\tilde{u},\tilde{u},\tilde{u})$ in $W$ (schrinking $U$ if necessary). Under this diffeomorphism, the map $\phi_{\tilde{u}}$ becomes $$\phi_{\tilde{u}} : B_n \times \tilde{U}  \times B_n \to \RR^n \times U \times \RR^n, \quad \phi_{\tilde{u}}(\xi,u,\eta)=(\xi',v,\eta'),$$ where each one of the $\xi', v, \eta'$ depends on $(\xi,u,\eta)$. This is a diffeomorphism onto a neighborhood of $(0,\tilde{u},0)$.
%Here $\xi'$ and $\eta'$ of course depend on the triple $(\xi,u,\eta)$. 
Once we restrict properly the domains, the  properties in Definition \ref{dfn:bisub2} are equivalent to: 
\begin{align}
\label{0u0}\Phi(0,u,0)=(0,u,0) \text{ for all  }u\in \tilde{U} \\
\label{ab}\alpha_{\xi}(u)=\beta_{\eta'}(v) \text{ and } \beta_{\eta}(u) = \alpha_{\xi'}(v) \text{ for all  $u\in \tilde{U}$  and $\xi,\eta$ small enough}.
\end{align}

%\mcomment{as in \S 3.3.3.  } 

After this preparation, we can start with the proof. Fix $\xi, \eta \in B_n$ and denote $$X := \sum_{i=1}^n \xi_i X_i \in \Gamma(U;\ker d\bs), \qquad Y:=\sum_{j=1}^n \eta_j Y_j \in \Gamma(U;\ker d\bt)$$ 
\noindent
\underline{\textbf{Claim}}: The vector $[X,Y](\tilde{u})$ lives in $\ker d_{\tilde{u}}s + \ker d_{\tilde{u}}t$. 
\\ \\ 
To prove this, let us denote $\varphi^X_t$ and $\varphi^Y_t$ the time-$t$ flows of $X$ and $Y$ respectively. For the moment, also fix $t\in \RR$ and put $u_t = \varphi^Y_t \phi^X_t(\tilde{u})$. Then $\varphi^X_{-t}(u_t) = \varphi^X_{-t}\varphi^Y_t\varphi^X_t(\tilde{u})$ and $\varphi^Y_{-t}(u_t)=\varphi^X_t(\tilde{u})$. We have $$(\varphi^X_{-t}(u_t),u_t,\varphi^Y_{-t}(u_t)) \in U \times_s U \times_t U$$ In coordinates (\ie applying $\chi$), this element becomes $$w = (-t\xi,u_t,-t\eta).$$
 Applying the map $\phi_{\tilde{u}}$, we find $v_t \in U$ and $\tilde{\xi}_t,\tilde{\eta}_t \in B_n$ such that 
 $$\phi_{\tilde{u}}(-t\xi,u_t,-t\eta) = (\tilde{\xi}_t,v_t,\tilde{\eta}_t)$$  
By Eq. \eqref{0u0} we have $ \tilde{\xi}_0=0$ and $\tilde{\eta}_0=0$.
Therefore there are smooth curves $\xi'_t$ and $\eta_t'$ in $\RR^n$ such that $\tilde{\xi}_t= t\xi'_t$ and   $\tilde{\eta}_t=t\eta'_t$, \ie 
 $$\Phi(-t\xi,u_t,-t\eta) = (t\xi'_t,v_t,t\eta'_t).$$
 So Eq. \eqref{ab} implies 
 \begin{equation}
\label{eq:abt}\alpha_{t\xi'_t}(v_t)=\beta_{-t\eta}(u_t) \text{ and } \beta_{t\eta'_t}(v_t)=\alpha_{-t\xi}(u_t).
\end{equation}
  Since $u_t = \beta_{t\eta}\alpha_{t\xi}(\tilde{u})$, the first of Eq. \eqref{eq:abt} gives: 
\begin{gather*}
\alpha_{t\xi'_t}(v_t)=\beta_{-t\eta}(u_t)=\beta_{-t\eta}\beta_{t\eta}\alpha_{t\xi}(\tilde{u}) = \alpha_{t\xi}(\tilde{u}) \\
\Rightarrow v_t=\alpha_{-t\xi'_t}\alpha_{t\xi}(\tilde{u})
\end{gather*}
The latter equation, together with the second of Eq. \eqref{eq:abt} imply 
\begin{eqnarray}\label{eqn1}
\beta_{-t\eta}\alpha_{-t\xi}\beta_{t\eta}\alpha_{t\xi}(\tilde{u}) = 
\beta_{-t\eta}\alpha_{-t\xi}(u_t) = 
\beta_{-t\eta}\beta_{t\eta'_t}(v_t) = \beta_{-t\eta}\beta_{t\eta'_t}\alpha_{-t\xi'_t}\alpha_{t\xi}(\tilde{u}) 
\end{eqnarray}
%\mcomment{For the the "intuitive" use of $\phi$ is that it allows to switch and $\alpha$ and a $\beta$ on the LHS of the abpve equation.}
%\mcomment{From here on the argument of why the "claim" is not necessary}

Since the $X_i$ are pairwise commuting and the same holds for the $Y_i$, this equals $\beta_{t(-\eta+\eta'_t)}\alpha_{t(\xi-\xi'_t)}(\tilde{u})$.
Write 
$$X_t:=\sum_{i=1}^n (\xi_i-(\xi'_t)_i) X_i \in \Gamma(\tilde{U};\ker d\bs)\text{ and } Y_t:=\sum_{i=1}^n (-\eta_i+(\eta'_t)_i) Y_i\in \Gamma(\tilde{U};\ker d\bt).$$ 
Due to the Baker-Campbell-Hasudorff formula, the expression $\beta_{t(-\eta+\eta'_t)}\alpha_{t(\xi-\xi'_t)}(\tilde{u})$ is exactly the time-1 flow $exp_{\tilde{u}}(Z_t)$ for the time-dependent vector field $$Z_t=tX_t+tY_t+[tX_t,tY_t]+\dots=t(X_t+Y_t)+o(t^2)$$

Now let $t$ vary.
Taking the second time derivative at zero of eq. \eqref{eqn1} we obtain
$[X,Y]_{\tilde{u}}=\frac{1}{2}\frac{d^2}{dt^2}|_{t=0}exp_{\tilde{u}}(Z_t)$.
%where $exp(Z_t)$ is the time-1 flow of the time-dependent vector field $Z_t$ \mcomment{I remark this because until here everything is time-independent}
%\mcomment{Should we cite Goldberg, or is there a better reference?}

%\mcomment{the simples example for the lemma is $Z_t=t\partial_x$ on $\RR$}
\begin{lemma}\label{lem:accvel}
$$\frac{d^2}{dt^2}|_{t=0}exp_{\tilde{u}}(Z_t) =\frac{d}{dt}|_{t=0}Z_t(\tilde{u}).$$
\end{lemma}
\begin{proof} Consider the curve $t\mapsto \gamma(t):=exp_{\tilde{u}}(Z_t)$, \ie the unique curve satisfying $\gamma(0)=\tilde{u}$ and $
\frac{d}{dt}\gamma(t)=(Z_t)(\gamma(t))$. We have $\frac{d}{dt}|_{t=0}\gamma(t)=0$, so the acceleration $\frac{d^2}{dt^2}|_{t=0}\gamma(t)$
defines a  well-defined element of $T_{\tilde{u}}U$. More precisely, it
 is the tangent vector corresponding to the derivation $$C^{\infty}(U)\to \RR, f\mapsto \frac{d^2}{dt^2}|_{t=0}(f\circ\gamma)(t).$$

To make this explicit we compute
\begin{align}
\frac{d^2}{dt^2}|_{t=0}(f\circ exp_{\tilde{u}}(Z_t))&=
 \frac{d}{dt}|_{t=0}  f_*((Z_t)(\gamma(t)))\\
 &=\frac{d}{dt}|_{t=0}[ f_*((Z_t)(\gamma(0)))+
 f_*((Z_0)(\gamma(t)))]\\
 &= \frac{d}{dt}|_{t=0} f_*((Z_t)(\tilde{u}))\\
 &=f_*\big(\frac{d}{dt}|_{t=0} (Z_t)(\tilde{u})\big)
\end{align}
 where in the second last equality we used $Z_0=0$. 
 \end{proof}

We compute easily $\frac{d}{dt}|_{t=0}Z_t(\tilde{u})=(X_0+Y_0)(\tilde{u})$, which lies in  $(\ker d_{\tilde{u}}s+\ker d_{\tilde{u}}t)$. As the Lie bracket is a derivation (w.r.t. multiplication by functions) in each entry, and
 the point $\tilde{u}\in U$ is arbitrary, we conclude that 
$\cF_U = \Gamma_c(U;\ker d\bs) + \Gamma_c(U;\ker d\bt)$ is a foliation (stable by the Lie bracket).

\begin{remark}
An example for Lemma \ref{lem:accvel} is the vector field $Z_t = t\partial_x$ on $\R$.
\end{remark}

%\bibliographystyle{halpha} 
%%\bibliography{paper}

%% \bibliographystyle{habbrv} 
%\bibliography{AlgFBib}

\end{document}